\documentclass[journal]{IEEEtran}

\usepackage[utf8]{inputenc} 
\usepackage{hyperref}       
\usepackage{url}            
\usepackage{booktabs}       
\usepackage{amsfonts}       
\usepackage{nicefrac}       
\usepackage{microtype}      
\usepackage{amsmath,amsfonts,amsthm,amssymb}
\usepackage{color}
\usepackage{algorithm}
\usepackage{algorithmic}
\usepackage{graphicx}
\usepackage{caption}
\usepackage{subcaption}

\newtheorem{theorem}{Theorem}[section]
\newtheorem{lemma}{Lemma}[section]
\newtheorem{corollary}{Corollary}[section]
\newtheorem{assumption}{Assumption}[section]

\newcommand{\ignore}[1]{}

\title{Decentralized Asynchronous Non-convex Stochastic Optimization on Directed Graphs}
 \author{Vyacheslav Kungurtsev, Mahdi Morafah, Tara Javidi~\IEEEmembership{Member,~IEEE,} Gesualdo Scutari~\IEEEmembership{Member,~IEEE,}
 \thanks{V. Kungurtsev are with Department of Computer Science, Czech Technical University in Prague}
 \thanks{M. Morafah and T. Javidi are with the Department of Electrical Engineering, University of California-San Diego}
 \thanks{G. Scutari is with the School of Industrial Engineering and the School of Electrical and Computer Engineering (ECE) at Purdue University}}
\begin{document}

\maketitle
\begin{abstract}
Distributed Optimization is an increasingly important
    subject area with the rise of multi-agent control and optimization. We consider a decentralized stochastic  optimization problem where the agents on a graph aim to asynchronously optimize a collective (additive) objective function consisting of  agents' individual (possibly non-convex) local objective functions. Each agent only has access to a noisy estimate of the gradient of its own function (one component of the sum of objective functions). We proposed an asynchronous distributed algorithm for such a class of problems. The algorithm combines  stochastic gradients with tracking in an asynchronous   push-sum framework and obtain the standard sublinear convergence rate for general non-convex functions, matching the rate of centralized stochastic gradient descent SGD. 
    Our experiments on a non-convex image classification task using convolutional neural network validate the convergence of our proposed algorithm across different number of nodes and graph connectivity percentages. 
\ignore{
Distributed Optimization is an increasingly important
    subject area with the rise of multi-agent control and optimization. We consider a decentralized optimization problem where the agents on a graph aim to asynchronously optimize a collective (additive) objective function consisting of  agents' individual (possibly non-convex) local objective functions. In particular, we propose a .... Combining ..., we obtain a sublinear convergence rate for general (non-convex) functions while we obtain a faster (I would specify what you mean by faster?) convergence rate in the convex setting. }
   
\end{abstract}

\section{Introduction}
In this paper we consider the global optimization problem,
\begin{equation}\label{eq:prob}
\min_{\mathbf{x}\in\mathbb{R}^n} \mathbb{E}_\xi\left[F(\mathbf{x},\xi)\right]\triangleq \sum_{i=1}^m \mathbb{E}_\xi\left[f_i(\mathbf{x},\xi)\right]
\end{equation}
where the agents' local objective functions $\mathbb{E}[f_i(\cdot,\xi)], i=1, \ldots, m$ are smooth and generally nonconvex and known only locally to the $m$ networked agents. In addition, each agent is assumed to only have access to noisy estimates of $f_i$ and its gradients.
Communication across the network is performed asynchronously in a gossip fashion, i.e., 
there is a (possibly) directed  graph $\mathcal{G}=\{\mathcal{V},\mathcal{E}\}$, 
and for each edge $e\in\mathcal{E}\subseteq \mathcal{V}\times\mathcal{V}$ the vertices elements $\{i,j\}\in e$ implies that node $i$ can send information to node $j$. Define $\mathcal{N}^{in}_i:=\{j\in\mathcal{V}: (j,i)\in\mathcal{E}\}$ and $\mathcal{N}^{out}_i:=\{j\in\mathcal{V}: (i,j)\in\mathcal{E}\}$.

We make the following assumption on the problem,
\begin{assumption}\label{as:problem}
\begin{enumerate}
    \item For a.e. $\xi$, $f_i(\cdot,\xi)$ is proper, closed, and $L_i$-Lipschitz continuously differentiable. Furthermore (for a.e. $\xi$) $F(\cdot,\xi)$ is bounded from below.
\item The (di)graph $\mathcal{G}$ is strongly connected.
\end{enumerate}
\end{assumption}

In this paper we uniquely address several concomitant challenges 1) the objective function is stochastic nonconvex 2) the sum-components of the objective function are known only locally to each agent 3) communication is fully asynchronous, modeled as each iteration consisting of a random activation among the agents
as well as delays in the communicated information and 4) the topology of the network is arbitrary (i.e., no hierarchical structure) and communication is \emph{directed}, i.e., that agent $i$ being able to send data to $j$ does not necessarily imply that $j$ can also send information to $i$.

\textbf{Previous Works} There are a number of works in the literature that consider \emph{distributed} or \emph{decentralized} stochastic optimization addressing a partial subset of these challenges. 

A bulk of works consider distributing computation in a shared memory setting while allowing for asynchronous updating (and thus read and/or write lock-free) including the classic~\cite{tsitsiklis1986distributed} and the seminal work~\cite{lian2015asynchronous}, the general framework for block/coordinate parallel updates given in~\cite{peng2016arock}, and many thereafter. These, however, assume that every computing node has access to the entire function, or a noisy estimate thereof, rather than a component of it, and does not consider communication across an arbitrary network.

A standard structure for distributed optimization is the ``hub-spoke", ``parameter-server" or ``master-slave" architecture. This is considered, for instance, in~\cite{agarwal2011distributed} and~\cite{mcmahan2017communication}. In this case, it is assumed that the nodes have a hierarchical structure of communication, with one node aggregating information and coordinating the computation to be performed across the other nodes, which communicate solely with the central node and not at all with each other. In this paper we consider a more general arbitrary graph topology modeling the communication links across agents.

Schemes that consider an arbitrary graph topology include~\cite{xin2019variance} which considers \emph{convex} problems and uses variance reduction (i.e., accessing the entire local gradient vector periodically), and~\cite{tang2018d} which only considers \emph{undirected} graphs and with synchronization barriers.

Moving closer to our setting, contemporary works considering a decentralized graph communication structure and asynchronous communication include~\cite{lian2017asynchronous}, which analyzes
nonconvex problems and ~\cite{ConvexAsynchDecSGD} considers convex ones. However, they consider only \emph{undirected} graphs, and furthermore the ultimate function being minimized is not the desired objective, but a scaled one, based on the frequency of updates of each agent. This can frequently
not be known a priori, and thus is a poor target for the objective function. Without this knowledge, any solution
to the problem would be biased.

The push-sum framework was introduced in~\cite{tsianos2012push} to avoid systematic bias in the solution of multi-agent optimizations problems on directed graphs. The analysis of distributed consensus with delays was first given in~\cite{nedic2010convergence}, who introduced \emph{virtual} nodes which model information as passing from one to the next as one less delay until it arrives at the real-time node. Note that these are purely theoretical instruments, and need not be stored.


In~\cite{olshevsky2018robust} a stochastic gradient algorithm is presented based on the push sum approach to handling optimization over a network with asynchronous directed communication. However, convergence is only proven for strongly convex problems. 

The paper~\cite{tian2018achieving} describes an algorithm for using the push-sum framework in a nonconvex setting with asynchronous parallel communication, for deterministic objectives. 

The paper~\cite{zhang2019fully} considers asynchronous communication and directed graphs, and presents an algorithm with provably linear convergence towards the optimum under the Polyak-\L ojasiewicz condition. Finally~\cite{zhang2019decentralized} considers nonconvex
federated learning using gradient tracking, and prove
convergence, consensus, and asymptotic agreement
of each agent's average gradient estimate. They
consider the synchronous setting and undirected graphs. 

\textbf{Contributions} In this paper we study the theoretical and numerical convergence properties of decentralized stochastic nonconvex optimization on directed graphs with asynchronous communication. Thus, this paper extends the work of~\cite{tian2018achieving} to consider noisy function data and~\cite{olshevsky2018robust} to the case of nonconvex objectives, closing an important gap in the literature for decentralized stochastic optimization.

\section{Algorithm}
The algorithm is presented as Algorithm~\ref{alg:pushsum}, and described below. 

All agents update asynchronously and continuously without coordination,  using noisy gradient estimates and  possibly delayed information from their neighbors. 
Each agent $i$ maintains and updates  the  following local variables: i) a local estimate $\mathbf{x}_{i}$ of the common optimization vector $\mathbf{x}$; ii) the auxiliary variable $\mathbf{z}_{i}$, aiming at tracking the sample gradient $ \tilde \nabla F$   of the sum-loss (we use $\tilde \nabla F$ a sample instance of $\nabla F$), not available locally; and iii) some  mass counters $\rho_{ij}$ and buffer variables $\tilde{{\rho}}_{ij}$,  
 $j \in \mathcal{N}_i^{in}$, which are instrumental to track properly the sum-gradient $ \nabla F(\cdot,\xi)$ in the presence of asynchrony (their update is commented below). The   $k$-th iterate of the above variables is denoted by   $\mathbf{x}_{i}^k$, $\mathbf{z}_{i}^k$, $\rho_{ij}^k$, and  $\tilde{{\rho}}_{ij}^k$, respectively. 
In  Algorithm~\ref{alg:pushsum}, the iteration index $k$ is understood as a {\it global} iteration counter $k$, unknown to the agents,  which increases by $1$ whenever a variable of the agents changes.
	Let $i^k$ be the agent triggering iteration $k \to k+1$; it executes Steps (S1)-(S3) (no necessarily within the same activation), as described below.
	
\texttt{(S1) Stochastic gradient step:}  The active agent $i^k$  updates its local variable $\mathbf{x}^k_{i^k}$ by moving along the direction  of the sample gradient estimate  $\mathbf{z}_{i^k}^k$, with a step-size $\gamma\in (0,1]$, generating  $\mathbf{v}_{i^k}^{k+1}$.

\texttt{(S2) Consensus step with delays:} Agent $i^k$ may receive delayed variables from its in-neighbors $j \in \mathcal{N}_{i^k}^{in}$, whose iteration index is  $k - d_j^k$, where  $d_j^k\geq 0$ is the delay. 
	To perform its update, it  first  sorts the ``age'' of  all the received variables from agent $j$ since $k=0$, and then picks the most recently generated one. This is implemented  maintaining a local counter $\tau_{i^k j}$, updated  recursively as $\tau_{i^k j}^k = \max(\tau_{i^k j}^{k-1}, k-d_j^k)$. Thus,  the  variable agent $i^k$ uses from $j$ has iteration index $\tau_{i^k j}^k$.  Given this (outdated) information, agent $i^k$ performs  a consensus update  with mixing matrix $\mathbf{W} = (w_{ij})_{i,j=1}^I$ (to be properly chosen, see Assumption \ref{as:alg} below), generating   $\mathbf{x}^{k+1}_{i^k}$.
	
\texttt{(S3) Robust gradient tracking:} 	This step aims at tracking the sample sum-gradient $ \tilde\nabla F$ in the presence of asynchrony; it  builds on the  the  asynchronous sum-push scheme  introduced in   \cite{tian2018achieving} (note that \cite{tian2018achieving} does not deal with stochastic gradients), and works as follows. Each agent $i$ maintains  mass counters $\rho_{ji}$ associated  to $\mathbf{z}_i$  that record the cumulative mass generated by $i$  for $j \in \mathcal{N}_i^{out}$ since $k=0$; and transmits $\rho_{ji}$. 
	In addition, agent $i$ also maintains  buffer variables $\tilde{{\rho}}_{ij}$
	to track the latest mass counter $\rho_{ij}$  from $j \in \mathcal{N}_i^{in}$ that has been used in its update.
The update of the $\mathbf{z}$- and $\rho$-variables  employed by agent $i^k$ is as follows. Agent $i^k$   first performs the sum step \texttt{(S3.1)} using  a possibly delayed mass counter $\rho_{i j}^{\tau_{i j}^k}$ received from $j$. By computing the difference  $\rho_{i^k j}^{\tau_{i^k j}^k} - \tilde{\rho}_{i^kj}^k$, it collects the sum of the $a_{i^kj} z_j$'s generated by $j$ that it has not yet added. Then, agent $i^k$   sums them together with the gradient correction term   $\tilde \nabla f_{i^k} (\mathbf{x}_{i^k}^{k+1},\zeta^k) - \tilde \nabla f_i (\mathbf{x}_{i^k}^{k},\zeta^{j(i^k,k)})$  to its current state variable $\mathbf{z}_{i^k}^k$  to form the intermediate mass $\mathbf{z}_{i^k}^{k+\frac{1}{2}}$, where  $j(i^k,k)$ is the last iteration $j$ before $k$ for which $i^k$ is the chosen agent. Next, in the push step \texttt{(S3.2)}, agent $i^k$ splits  $\mathbf{z}_{i^k}^{k+\frac{1}{2}}$, maintaining $a_{i^k i^k} \mathbf{z}_{i^k}^{k+\frac{1}{2}}$ for itself and accumulating  $a_{ji^k} \mathbf{z}_{i^k}^{k+\frac{1}{2}}$ to its local mass counter $\rho_{ji^k}^k$, to be transmit to $j \in \mathcal{N}_{i^k}^{out}$.
Since the last mass counter agent $i^k$   processed is $\rho_{i^k j}^{\tau_{i^k j}^k}$, it sets $\tilde{\rho}_{i^kj} =\rho_{i^k j}^{\tau_{i^k j}^k} $.



\begin{algorithm}[h!]
	\caption{\textbf{- Asynchronous Stochastic Gradient Descent with Tracking}}
	\label{alg:pushsum}
	\begin{algorithmic}
		\STATE{\textbf{Initialization:}  Set $k=0$, Set $\mathbf{x}^0_i=\mathbf{0}$ and $\mathbf{z}_i^0=\tilde f_i(\mathbf{0},\xi^0)$ for all $i$.}\\
		\WHILE{Not converged}
		\STATE{Choose $(i^k,d^k)$;}
		\STATE{Set $\tau^k_{i^k j}=\max\{\tau_{i^k j}^{k-1},k-d^k_j\},\quad \forall j\in \mathcal{N}_{i^k}^{in}$};
		\STATE{\texttt{(S1) (Stochastic gradient update):} Set $\mathbf{v}^{k+1}_{i^k}=\mathbf{x}^k_{i^k}-\gamma^k \mathbf{z}^k_{i^k}$}
		\STATE{\texttt{(S2) Consensus (with delayed info):} $\mathbf{x}_{i^k}^{k+1}=
 {w}_{i^ki^k} \mathbf{v}^{k+1}_{i^k}+\sum_{j\in \mathcal{N}^{in}_{i^k}}  {w}_{i^k j} 
\mathbf{v}_{j}^{\tau^k_{i^kj}}$}
		\STATE{\texttt{(S3) Robust gradient tracking:}}\smallskip 
	\STATE{\texttt{\hspace{0.7cm}(S3.1) Sum step}:}
\[
\begin{array}{l}
\mathbf{z}^{k+\frac12}_{i^k}=\mathbf{z}^k_{i^k}+\sum_{j\in\mathcal{N}^{in}_{i^k}}(\rho^{\tau^k_{i^k j}}_{i^k j}-\tilde{\rho}^k_{i^k j})\\ \qquad +\tilde \nabla f_{i^k} (\mathbf{x}_{i^k}^{k+1},\zeta^k) - \tilde \nabla f_i (\mathbf{x}_{i^k}^{k},\zeta^{j(i^k,k)});
\end{array}
\]
		\STATE{\hspace{0.8cm}\texttt{(S3.2) Push step:}}
\[
\begin{array}{l}
\mathbf{z}^{k+1}_{(i^k)}=a_{i^k i^k} \mathbf{z}^{k+\frac12}_{(i^k)};\\
\mathbf{\rho}^{k+1}_{ji^k}=\mathbf{\rho}^k_{ji^k}+a_{ji^k} \mathbf{z}_{i^k}^{k+\frac12},\, \forall j\in\mathcal{N}^{out}_{i^k};
\end{array}
\]
		\STATE{\hspace{0.8cm}\texttt{(S3.3) Mass-Buffer update:}}
\[
\tilde{\mathbf{\rho}}^{k+1}_{i^k j} = \mathbf{\rho}^{\tau^k_{i^k j}}_{i^k j},\,\forall j\in\mathcal{N}^{in}_{i^k};
\]
	\STATE{\texttt{(S4):}
Untouched state variables shift to state $k+1$ while keeping the same value; $k\leftarrow k+1.$}
		\ENDWHILE
	\end{algorithmic} 
\end{algorithm} 

We make the following Assumption regarding the communication network, activation, delays, and stochastic gradient estimates.
\begin{assumption}\label{as:alg}

\begin{enumerate}
\item It holds that there exists an $\bar{m}$ such that for all $i\in\mathcal{V}$, $w_{ij}\ge \bar{m}$ and $a_{ij}\ge \bar{m}$ for all $(i,j)\in\mathcal{E}$. Furthermore the matrix $\mathbf{W}$ composed of $w_{ij}$ is row-stochastic ($\mathbf{W}\mathbf{1}=\mathbf{1}$) and $\mathbf{A}$ composed of $a_{ij}$ is column-stochastic ($\mathbf{A^T}\mathbf{1}=\mathbf{1}$).
\item There is a $T\in\mathbb{R}^+$ such that the activations satisfy $\cup_{t=k}^{k+T-1} i^t =\mathcal{V}$.
\item There is a $D\in \mathbb{R}^+$ such that the delays satisfy $0\le d_j^k\le D$ for all $j\in\mathcal{N}^{in}_{i^k}$ for all $k\in\mathbb{N}$
    \item The assumptions on the stochastic estimate are the standard unbiased estimate with bounded variance conditions,
\begin{equation}\label{eq:stochest}
\begin{array}{l}
\mathbb{E}\left[\tilde\nabla f_{i^k}(\mathbf{x}^{k+1}_{(i^k)},\zeta^k)\right] = \nabla f_{i^k}(\mathbf{x}^{k+1}_{(i^k)}), \\
\mathbb{E}\left[\left\|\tilde\nabla f_{i^k}(\mathbf{x}^{k+1}_{(i^k)},\zeta^k)-\nabla f_{i^k}(\mathbf{x}^{k+1}_{(i^k)})\right\|^2\right] = \sigma^2
\end{array} 
\end{equation}
\end{enumerate}
\end{assumption}



\section{Convergence}
In this section we prove the convergence properties of
Algorithm~\ref{alg:pushsum} for stochastic nonconvex objectives. We begin introducing some intermediate results, instrumental for our proofs. 

\subsection{Preliminaries}
Following \cite{tian2018achieving}, we define augmented variables $\mathbf{h}^k\triangleq \left[(\mathbf{x}^k)^T\, (\mathbf{v}^k)^T\, (\mathbf{v}^{k-1})^T\,
...(\mathbf{v}^{k-D})^T\right]$, where $D$ is the maximum possible delay time.
We denote the augmented gradient estimate stacked vector as $\hat{\mathbf{z}}^k$. Ultimately, there is a matrix $\hat{\mathbf{A}}^k$ representing the mixing of $\hat{z}$, i.e., $\hat{\mathbf{z}}^{k+1}=\hat{\mathbf{A}}^k \hat{\mathbf{z}}^k+\mathbf{p}^k$ \cite{tian2018achieving},
where $\mathbf{p}^k$ is simply the stacked change in the vector from the new stochastic gradient updates.   For the consensus of the expanded model vector $\mathbf{h}^k$ we denote by $\hat{\mathbf{W}}^k$ the corresponding mixing matrix, i.e., $\mathbf{h}^{k+1}=\hat{\mathbf{W}}^k (\mathbf{h}^k+\mathbf{\delta}^k)$, 
with $\mathbf{\delta}^k$ defining the stacked update vector. The   explicit expressions of the matrices $\hat{\mathbf{A}}^k$ and $\hat{\mathbf{W}}^k$ are immaterial for our   subsequent convergence analysis; all it is needed are their mixing rate properties, as recalled next. 



\begin{lemma}\cite[Lemma 14]{tian2018achieving}
In the setting of Algorithm 1, there exists a sequence of stochastic vectors 
$\{\mathbf{\xi}^k\}$ such that, for any $k\ge t\in\mathbb{N}$ and
$i,j\in \mathcal{V}$, there holds 
\[
|\hat{\mathbf{A}}_{ij}^{k:t}-\xi^k_i|\le C\rho^{k-t},
\]
for some   $C>0$ and $\rho<1$.
\end{lemma}
\begin{lemma}\cite[Lemma 16]{tian2018achieving}
In the setting of Algorithm 1, there exists a sequence of stochastic vectors 
$\{\mathbf{\psi}^k\}$ such that, for any $k\ge t\in\mathbb{N}$ and
$i,j\in \mathcal{V}$, there holds 
\[
|\hat{\mathbf{W}}^{k:t}-\mathbf{1}(\psi^t)^T|\le C\rho^{k-t},
\]
for some  $C>0$ and $\rho<1$.
\end{lemma}

Finally, we define a new vector $\bar{\mathbf{z}}^k_{i^k}$. This represents
the update that would be made if actual rather that stochastic
gradients were computed, i.e, $\bar{\mathbf{z}}_{(i)}^k=\nabla f_i(\mathbf{x}_i^0)+\sum_{t=k:\, i^k=i} (\nabla f_i(\mathbf{x}_{i}^{t+1})-f_i(\mathbf{x}_{i}^{t}))$. 

It can be seen that,
\begin{equation}\label{eq:zzbar}
\mathbb{E}\left[\bar{\mathbf{z}}^k_{i^k}-\mathbf{z}^k_{i^k}\right] = 0 \quad \text{and}\quad 
\mathbb{E}\left[\left\|\bar{\mathbf{z}}^k_{i^k}-\mathbf{z}^k_{i^k}\right\|^2\right] \le m\sigma^2.
\end{equation}

To study convergence of Algorithm 1, we introduce  the following error terms, defining the gradient
tracking, consensus, and gradient norm errors for the
evolving iterations:
\begin{equation}\label{eq:errterms}
\begin{array}{l}
E_t^k = \left\|\bar{\mathbf{z}}^k_{(i^k)}-\xi^{k-1}_{i^k} \left(\mathbf{1}\mathbf{1}^T\otimes \mathbf{I}_n\right) \bar{\mathbf{z}}^k\right\|^2,\\
E^k_c=\left\|\mathbf{h}^k-\mathbf{1}_{m}\otimes \mathbf{x}^k_\psi\right\|^2,\,E^k_z=\left\|\bar{\mathbf{z}}^k_{(i^k)}\right\|^2.
\end{array}
\end{equation}
Note that,
\[
\mathbb{E}|\mathbf{z}^k_{i^k}|\le \sqrt{E_z^k}+\mathbb{E}|\mathbf{z}^k_{i^k}-\bar{\mathbf{z}}^k_{i^k}| \le \sqrt{E_z^k}+\sqrt{m}\sigma.
\]

\subsection{Convergence Theory}

The proof of the main convergence theory begins similarly as in~\cite{tian2018achieving}, however, subsequently changes significantly in order to account for the noise and then also set up the possibility of deriving specific convergence rates for the optimization, consensus, and tracking errors.

\begin{theorem}\label{th:conv}
Let Assumptions~\ref{as:problem} and~\ref{as:alg} hold.

Assume that the stepsize sequence $\{\gamma^k\}$ satisfies,
\[
\begin{array}{l}
\sum\limits_{k=1}^\infty \gamma^k = \infty,\,\,\sum\limits_{k=1}^\infty (\gamma^k)^2 < \infty,
\\
\qquad \gamma^0 \le \min\left\{\frac{1}{1+\eta},\frac{1}{4(L+C^c_2+C^t_2)} \right\}
\end{array}
\]
where $C^c_2$ and $C^t_2$ are constants to be defined in the proof.

The merit function $M(\bar{\mathbf{z}}^k,\mathbf{h}^k,\mathbf{x}^k_\psi) = \mathbb{E}\left[E^k_t+E^k_c+E^k_z\right]$ is sublinearly
convergent with the standard ergodic rate,
\[
\sum\limits_{l=0}^k \gamma^l M(\bar{\mathbf{z}}^l,\mathbf{h}^l,\mathbf{x}^l_\psi)
\le \frac{C}{\sum\limits_{l=0}^k \gamma^l}
\]
for some constant $C>0$.
\end{theorem}
\begin{proof}
Consider the application of the Descent Lemma to $F$ applied at $x_\psi^k$ and $x_\psi^{k+1}$.
\[
\begin{array}{l}
\mathbb{E}\left[F(\mathbf{x}^{k+1}_{(\psi)})\right]\le 
\mathbb{E}\left[F(\mathbf{x}^k_{\psi})\right]+
\gamma^k \psi^k_{i^k}\mathbb{E}\left[\left\langle \nabla F(\mathbf{x}^k_\psi),-\mathbf{z}^k_{(i^k)}\right\rangle\right]
\\ \qquad\qquad +\frac{L(\gamma^k \psi^k_{i^k})^2}{2}\mathbb{E}\left[\left\|\mathbf{z}^k_{(i^k)}\right\|^2\right] \\ 
\qquad \le 
\mathbb{E}\left[F(\mathbf{x}^k_{\psi})\right]+
\gamma^k \psi^k_{i^k}\mathbb{E}\left[\left\langle \nabla F(\mathbf{x}^k_\psi),-\bar{\mathbf{z}}^k_{(i^k)}\right\rangle\right]
\\ \qquad\qquad +\gamma^k \psi^k_{i^k}\mathbb{E}\left[\left\langle \nabla F(\mathbf{x}^k_\psi),\bar{\mathbf{z}}^k_{(i^k)}-\mathbf{z}^k_{(i^k)}\right\rangle\right]
\\ \qquad\qquad +L(\gamma^k \psi^k_{i^k})^2\mathbb{E}\left[\left\|\bar{\mathbf{z}}^k_{(i^k)}\right\|^2+\left\|\mathbf{z}^k_{(i^k)}-\bar{\mathbf{z}}^k_{(i^k)}\right\|^2\right]\\
\qquad \le \mathbb{E}\left[F(\mathbf{x}^k_\psi)\right]+L(\gamma^k)^2\mathbb{E}\left[\left\|\bar{\mathbf{z}}^k_{(i^k)}\right\|^2\right]
\\ \qquad\qquad +\gamma^k\psi^k_{i^k}\mathbb{E}\left[\left\langle (\xi^{k-1}_{i^k})^{-1} \bar{\mathbf{z}}^k_{(i^k)},-\bar{\mathbf{z}}^k_{(i^k)}\right\rangle\right] \\\qquad\qquad
+\gamma^k \psi^k_{i^k}\mathbb{E}\left[\left\langle \left(\mathbf{1}\mathbf{1}^T\otimes \mathbf{I}_n \right)\bar{\mathbf{z}}^k-(\xi^{k-1}_{i^k})^{-1} \bar{\mathbf{z}}^k_{(i^k)},-\bar{\mathbf{z}}^k_{(i^k)}\right\rangle\right]\\\qquad\qquad
+\gamma^k \psi^k_{i^k}\mathbb{E}\left[\left\langle \nabla F(\mathbf{x}^k_\psi)- \left(\mathbf{1}\mathbf{1}^T\otimes \mathbf{I}_n \bar{\mathbf{z}}^k\right),-\bar{\mathbf{z}}^k_{(i^k)}\right\rangle\right] \\ \qquad \qquad+\gamma^k \psi^k_{i^k}\mathbb{E}\left[\left\langle \nabla F(\mathbf{x}^k_\psi),\bar{\mathbf{z}}^k_{(i^k)}-\mathbf{z}^k_{(i^k)}\right\rangle\right]\\ \qquad\qquad +L(\gamma^k \psi^k_{i^k})^2\mathbb{E}\left[\left\|\mathbf{z}^k_{(i^k)}-\bar{\mathbf{z}}^k_{(i^k)}\right\|^2\right]
\\
\qquad \le \mathbb{E}\left[F(\mathbf{x}^k_\psi)\right]+L(\gamma^k)^2\mathbb{E}\left[\left\|\bar{\mathbf{z}}^k_{(i^k)}\right\|^2\right]
\\ \qquad\qquad +\gamma^k\psi^k_{i^k}\mathbb{E}\left[\left\langle (\xi^{k-1}_{i^k})^{-1} \bar{\mathbf{z}}^k_{(i^k)},-\bar{\mathbf{z}}^k_{(i^k)}\right\rangle\right] \\\qquad\qquad\qquad
+\gamma^k \psi^k_{i^k}\left(\frac{\beta_1}{2}E^k_t+\frac{1}{2\beta_1} E^k_z\right)
\\\qquad\qquad\qquad
+\gamma^k \psi^k_{i^k}\left(\frac{\beta_1Lm}{2}E^k_c+\frac{1}{2\beta_1} E^k_z\right) \\ \qquad\qquad \qquad+\gamma^k \psi^k_{i^k}\mathbb{E}\left[\left\langle \nabla F(\mathbf{x}^k_\psi),\bar{\mathbf{z}}^k_{(i^k)}-\mathbf{z}^k_{(i^k)}\right\rangle\right]\\ \qquad\qquad +L(\gamma^k \psi^k_{i^k})^2\mathbb{E}\left[\left\|\mathbf{z}^k_{(i^k)}-\bar{\mathbf{z}}^k_{(i^k)}\right\|^2\right] 
\end{array}
\]
where in the last inequality we used the Cauchy-Schwartz and
Young's inequality,
$ab\le \frac{\beta}{2}a^2+\frac{1}{2\beta}b^2$, twice.

Set $\beta_1=\beta_2=2/\eta$, then,
\begin{equation}\label{eq:descbetaeta}
\begin{array}{l}
\mathbb{E}\left[F(\mathbf{x}^{k+1}_{(\psi)})\right] \le \mathbb{E}\left[F(\mathbf{x}^k_\psi)\right]-\left(\frac{\eta\gamma^k}{2}+L(\gamma^k)^2\right)
\mathbb{E}\left[\left\|\bar{\mathbf{z}}^k_{(i^k)}\right\|^2\right] \\ 
\qquad\qquad\qquad+\frac{\gamma^k}{\eta} E^k_t+\frac{\gamma^k}{\eta}E^k_c
\\ 
\qquad\qquad\qquad+\gamma^k \psi^k_{i^k}\mathbb{E}\left[\left\langle \nabla F(\mathbf{x}^k_\psi),\bar{\mathbf{z}}^k_{(i^k)}-\mathbf{z}^k_{(i^k)}\right\rangle\right]\\ \qquad\qquad \qquad+L(\gamma^k )^2m\sigma^2 
\end{array}
\end{equation}

Now it holds that by~\cite[Proposition 17]{tian2018achieving},
\begin{equation}\label{eq:errc-b-e}
\begin{array}{l}
\mathbb{E}\left[\sqrt{E^{k}_c}\right] \le C_2 \rho^k \mathbb{E}\left[\sqrt{E^0_c}\right]+C_2 \sum_{l=0}^{k-1} \rho^{k-l}\gamma^l 
\mathbb{E}|\mathbf{z}^l_{i^l}| \\ \qquad \le C_2 \rho^k \mathbb{E}\left[\sqrt{E^0_c}\right]+C_2 \sum_{l=0}^{k-1} \rho^{k-l}\gamma^l 
\left(\mathbb{E} \sqrt{E^l_z}+\sqrt{m}\sigma\right)
\end{array}
\end{equation}
which implies, by~\cite[Lemma 26]{tian2018achieving} that, after taking full expectations, that there exist $C^c_1$ and $C^c_2$ such
that,
\begin{equation}\label{eq:errc-b2-e}
\sum\limits_{l=0}^k \mathbb{E}\left[E^{l}_c\right] \le C^c_1 +C^c_2 \sum_{l=0}^{k} (\gamma^l)^2 \left(\mathbb{E}\left[E^l_z\right]+m\sigma^2\right)
\end{equation}

Similarly, from the proof of~\cite[Proposition 18]{tian2018achieving}, it can be seen that,
\begin{equation}\label{eq:errt-e}
\begin{array}{l}
\mathbb{E}\left[\sqrt{E^{k}_t}\right] \le 3 C_0 C_L\sum_{l=0}^{k-1}\rho^{k-l}\mathbb{E}\left[\sqrt{E^l_c}+\gamma^l(\sqrt{E^l_z}+\sqrt{m}\sigma)\right]\\ \qquad\qquad +C_0 \rho^k
\|\mathbf{z}^0\|
\end{array}
\end{equation}
and so, again as in~\cite[Lemma 26]{tian2018achieving} 
\begin{equation}\label{eq:errt2-e}
\sum\limits_{l=0}^k \mathbb{E}\left[E^{l}_t\right] \le C^t_1 +C^t_2 \sum_{l=0}^{k} (\gamma^l)^2 \left(\mathbb{E}\left[E^l_z\right]+m\sigma^2\right)
\end{equation}

Now, summing up~\eqref{eq:descbetaeta}, taking full expectations
we get, using~\eqref{eq:zzbar}
\begin{equation}\label{eq:sumconv1}
\begin{array}{l}
\sum\limits_{l=0}^k \left(\frac{\gamma^l\eta}{2}+L(\gamma^l)^2\right)\mathbb{E}\left\|\bar{\mathbf{z}}^l_{(i^l)}\right\|^2
\le F(\mathbf{x}^0)-F_m \\ \qquad +\sum\limits_{l=0}^k\left[\frac{\gamma^l}{\eta} \mathbb{E}[E^l_t]+\frac{\gamma^l}{\eta}\mathbb{E}[E^l_c]\right]+\sum\limits_{l=0}^k (\gamma^l)^2 L m\sigma^2
\end{array}
\end{equation}
Now add $\sum\limits_{l=0}^k\left[\gamma^l \mathbb{E}[E^l_t]+\gamma^l\mathbb{E}[E^l_c]\right]$ to both sides,
use $\frac{\gamma^l}{\eta}+\gamma^l \le 1$
and plug in~\eqref{eq:errc-b2-e} and~\eqref{eq:errt2-e} to get,
\[
\begin{array}{l}
\sum\limits_{l=0}^k \left(\frac{\gamma^l\eta}{2}-(L+C^c_2+C^t_2)(\gamma^l)^2\right)\mathbb{E}\left\|\bar{\mathbf{z}}^l_{(i^l)}\right\|^2 \\ \qquad\qquad
+\sum\limits_{l=0}^k \left[\gamma^l E^l_t+\gamma^l E^l_c\right] \\ \qquad
\le F(\mathbf{x}^0)-F_m +C_1^c+C_1^t+ C_\sigma \sum\limits_{l=0}^k(\gamma^l)^2
\end{array}
\]
and so for sufficiently small $\gamma^0$, we have,
\[
\sum\limits_{l=0}^k \gamma^l \left[E^l_z+ E^l_t+ E^l_c\right] \\ \qquad
\le F(\mathbf{x}^0)-F_m + C_\sigma m\sigma^2 \sum\limits_{l=0}^k(\gamma^l)^2
\]
and the proof claim follows from the general
stepsize conditions.
\end{proof}

\begin{corollary}
With the specific step-size choice of,
\[
\gamma^k = \frac{1}{k^\alpha},\, \alpha\in(1/2,1]
\]
we have the following convergence rates,
\[
\begin{array}{l}
M(\bar{\mathbf{z}}^k,\mathbf{h}^k,\mathbf{x}^k_\psi) = o\left(\frac{1}{k^{1-\alpha}}\right),\,
\mathbb{E}\left[E^k_t\right] = o\left(\frac{1}{k}\right),\, \\ \quad 
\mathbb{E}\left[E^k_c\right] = o\left(\frac{1}{k}\right),\,
\mathbb{E}\left[E^k_z\right] = o\left(\frac{1}{k^{1-\alpha}}\right)
\end{array}
\]
\end{corollary}
\begin{proof}
The first follows directly from Theorem~\ref{th:conv}, i.e., the right hand side is bounded and thus the sum on the left must be bounded, and so $\gamma^k M^k=o\left(\frac{1}{k}\right)$, thus the form of
$\gamma^k$ implies the rate for $M^k$. 

The rates for $\mathbb{E}\left[E^k_c\right]$
and $\mathbb{E}\left[E^k_t\right]$ follow from
~\eqref{eq:errc-b2-e} and~\eqref{eq:errt2-e}, respectively, as the right hand side is bounded
and thus the sum on the left must be. 

Finally the
bound for $\mathbb{E}\left[E^k_t\right]$ follows
from the finiteness of the right hand side of ~\eqref{eq:sumconv1} and the form of $\gamma^k$.
\end{proof}
We note how, similarly as in ~\cite{pu2019sharp}, the consensus errors converge quicker than the optimization and asymptotically the optimization dominates the overall convergence rate, in this case arbitrarily close to the standard SGD nonconvex rate of $O\left(\frac{1}{\sqrt{k}}\right)$~\cite{ghadimi2013stochastic}. 

\section{Experiments} 
In this section we aim to numerically study different aspects of our proposed Algorithm~\ref{alg:pushsum} on a non-convex optimization task.

The non-convex optimization task is image classification using neural network on MNIST \cite{lecun2010mnist} dataset where we aim to solving it using Algorithm~\ref{alg:pushsum} on a randomly generated digraph. We randomly shuffle the dataset and uniformly partition it across the nodes where the data partitions are disjoint. In our randomly generated digraph structure each agent $i$ sends its updates to $3$ out-neighbors; one of them is the next agent $i+1$ in the cycle order and the other two are selected randomly from a uniform distribution, following \cite{tian2018achieving}. We use the convolutional neural network (CNN) architecture used in Tensorflow tutorial \cite{Tensorflow-MNIST}. For each experiment we used a step-wise learning rate reduction schedule to achieve the best results. 
We selected the initial learning rate (step-size), step reduction interval, and size of reduction from the grid $[1, 0.8, 0.6, 0.4, 0.3, 0.2, 0.1]$, $[5000, 7500, 10000]$, and $[1.5, 1.6, 1.8, 2]$ respectively. Unless stated, otherwise we fixed the number of iterations that each node performs to $45000$ and stored results at every $30$ seconds. The experiments are done in Python environment using Tensorflow V2 and MPI (mpi4py) on RCI\footnote{The access to the computational infrastructure of the OP VVV funded project CZ.02.1.01/0.0/0.0/16\_019/0000765 ``Research Center for Informatics'' is also gratefully acknowledged.} clusters over the cpu nodes.


\textbf{Convergence and Scalability} The experiments in this section analyze the convergence property of our method for the described task. We perform experiments for $I = 2, 4, 8, 16$ nodes with a fixed graph connectivity of $0.7$. We report the results on the node-wise average parameters, i.e. $x^{avg}= \sum_i x_i$.



Figure \ref{fig1} shows the time-wise convergence results for different number of nodes. We can see that the accuracy drops monotonically as the number of nodes increased. This suggests for this scale we do not witness speedup with decentralized parallelism, although accurate training is still achievable. 

\begin{figure}[t]
    \centering
    \begin{subfigure}[b]{0.40\textwidth}
        \includegraphics[width=1\linewidth]{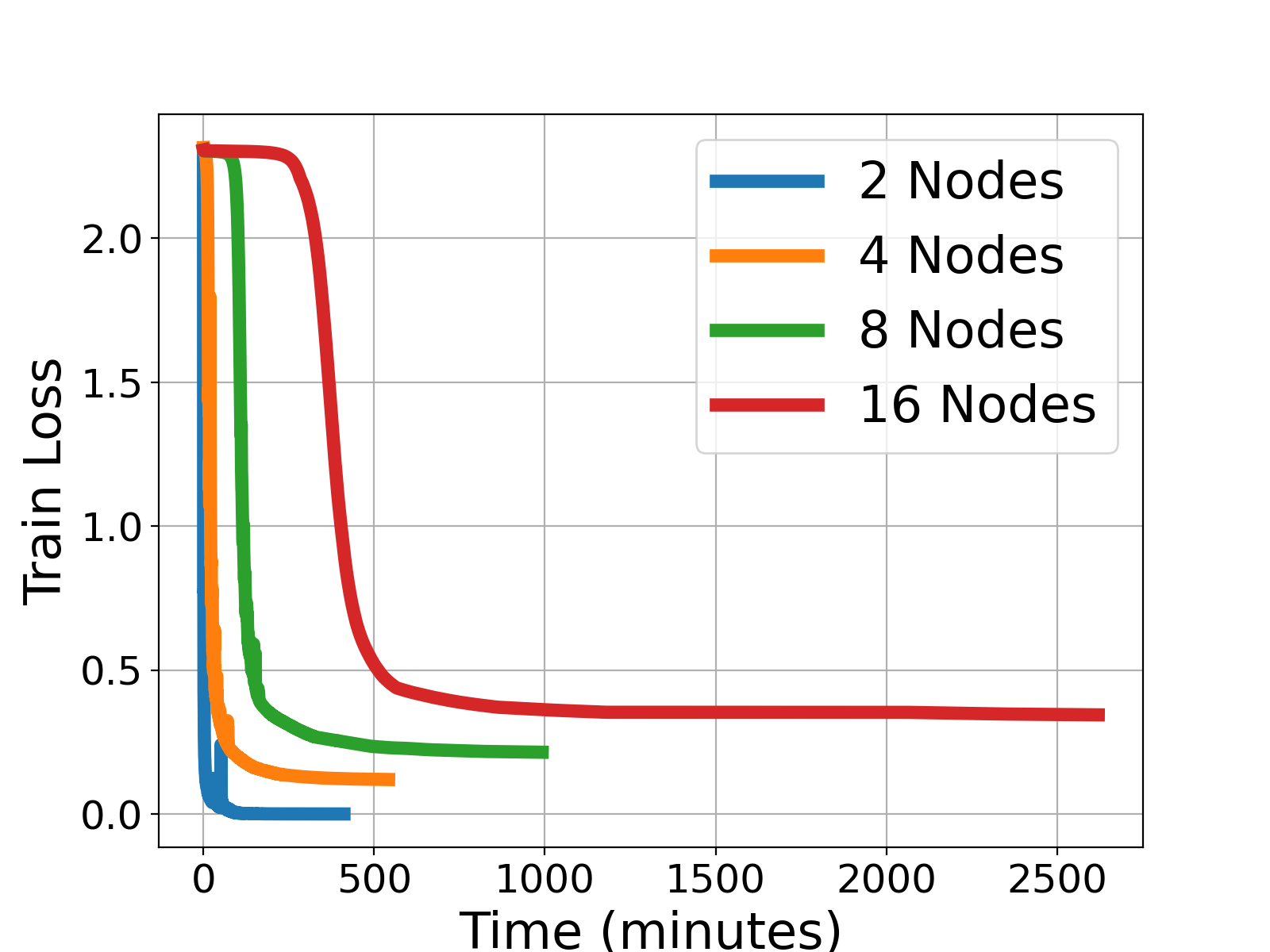}
        \label{fig1:a}
    \end{subfigure}
    \begin{subfigure}[b]{0.40\textwidth}
        \includegraphics[width=1\linewidth]{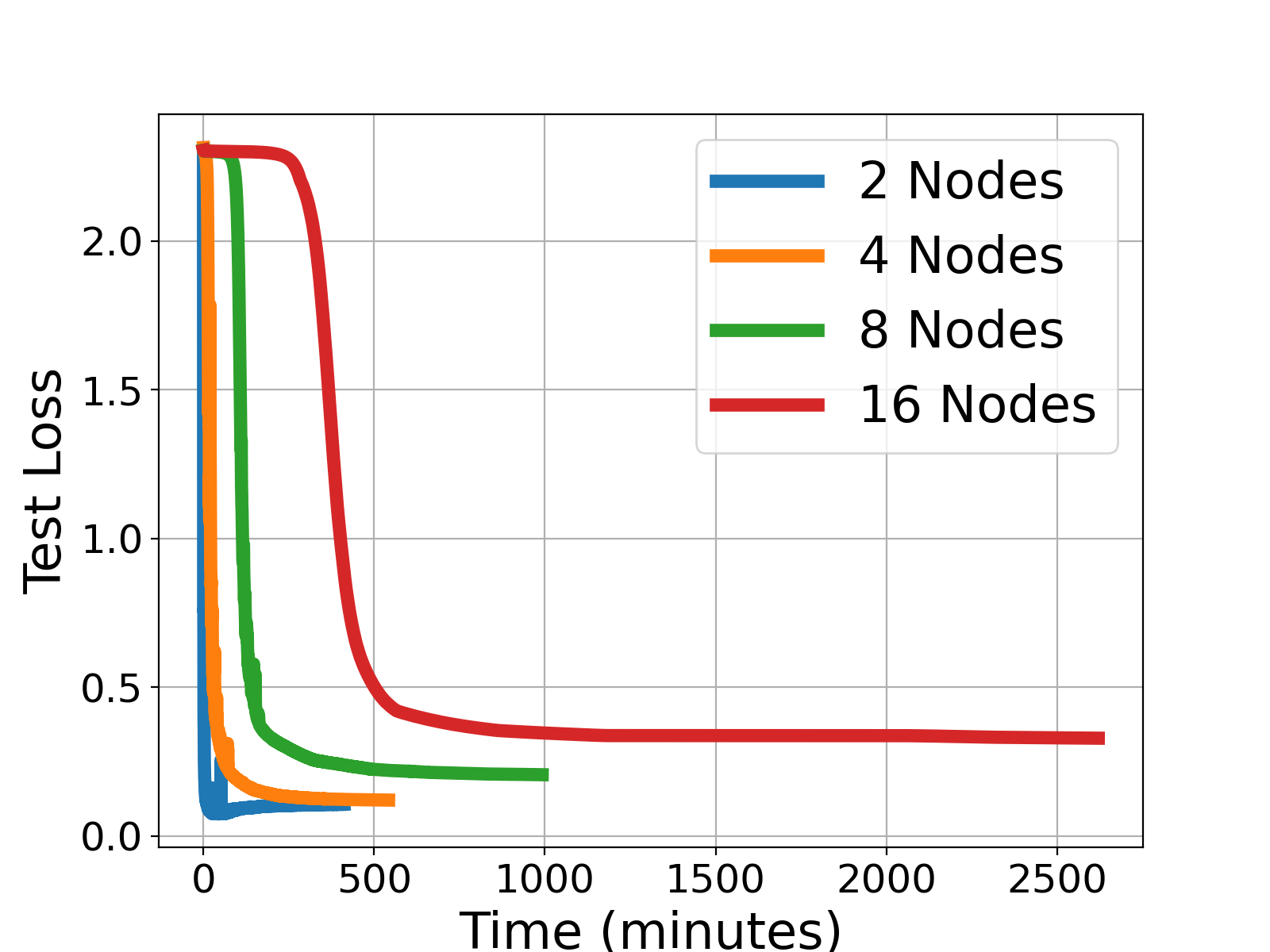}
        \label{fig1:b}
    \end{subfigure}
    \\ \vspace{0.1in}
    \begin{subfigure}[b]{0.40\textwidth}
        \includegraphics[width=1\linewidth]{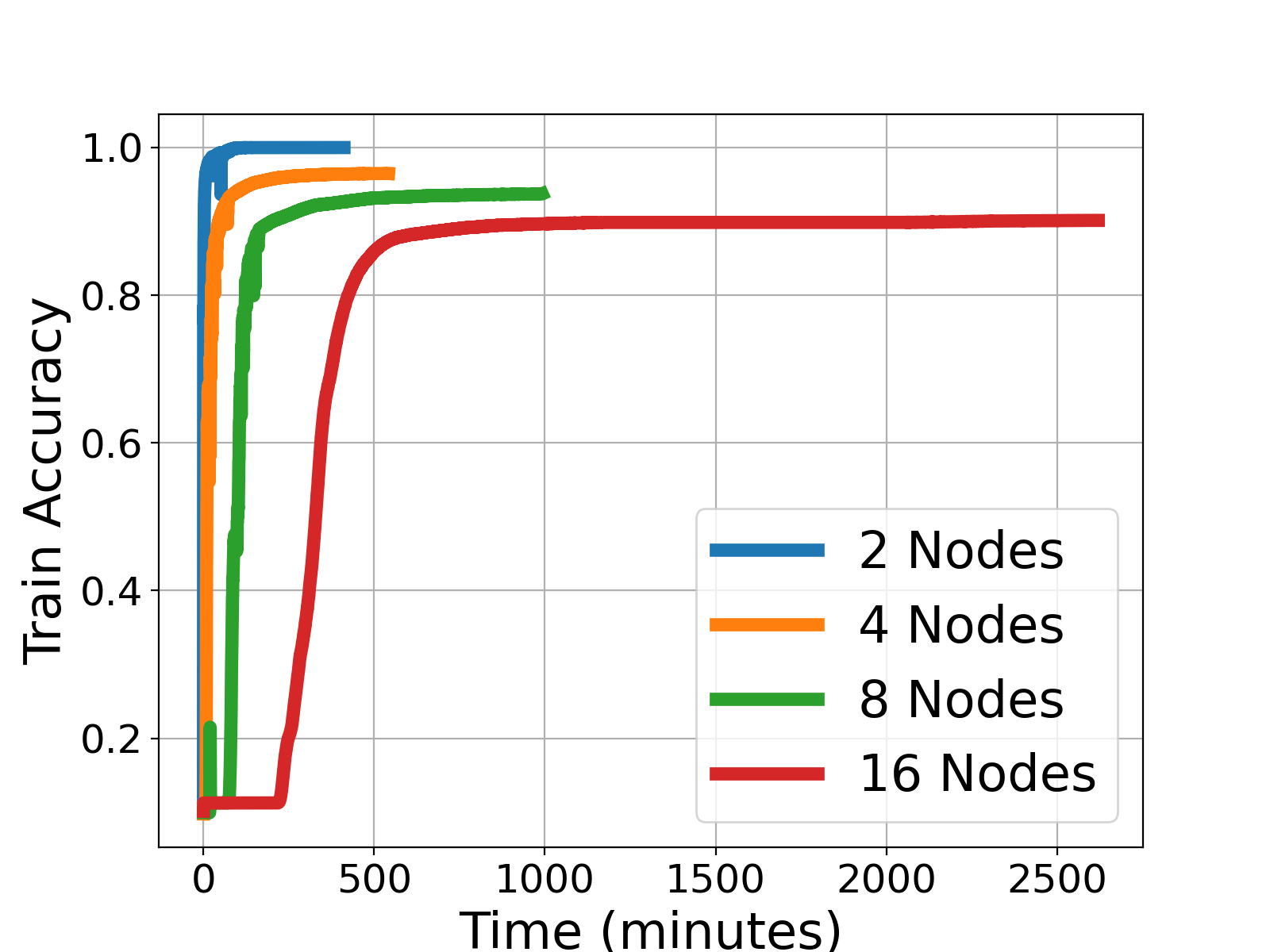}
        \label{fig1:c}
    \end{subfigure}
    \begin{subfigure}[b]{0.40\textwidth}
        \includegraphics[width=1\linewidth]{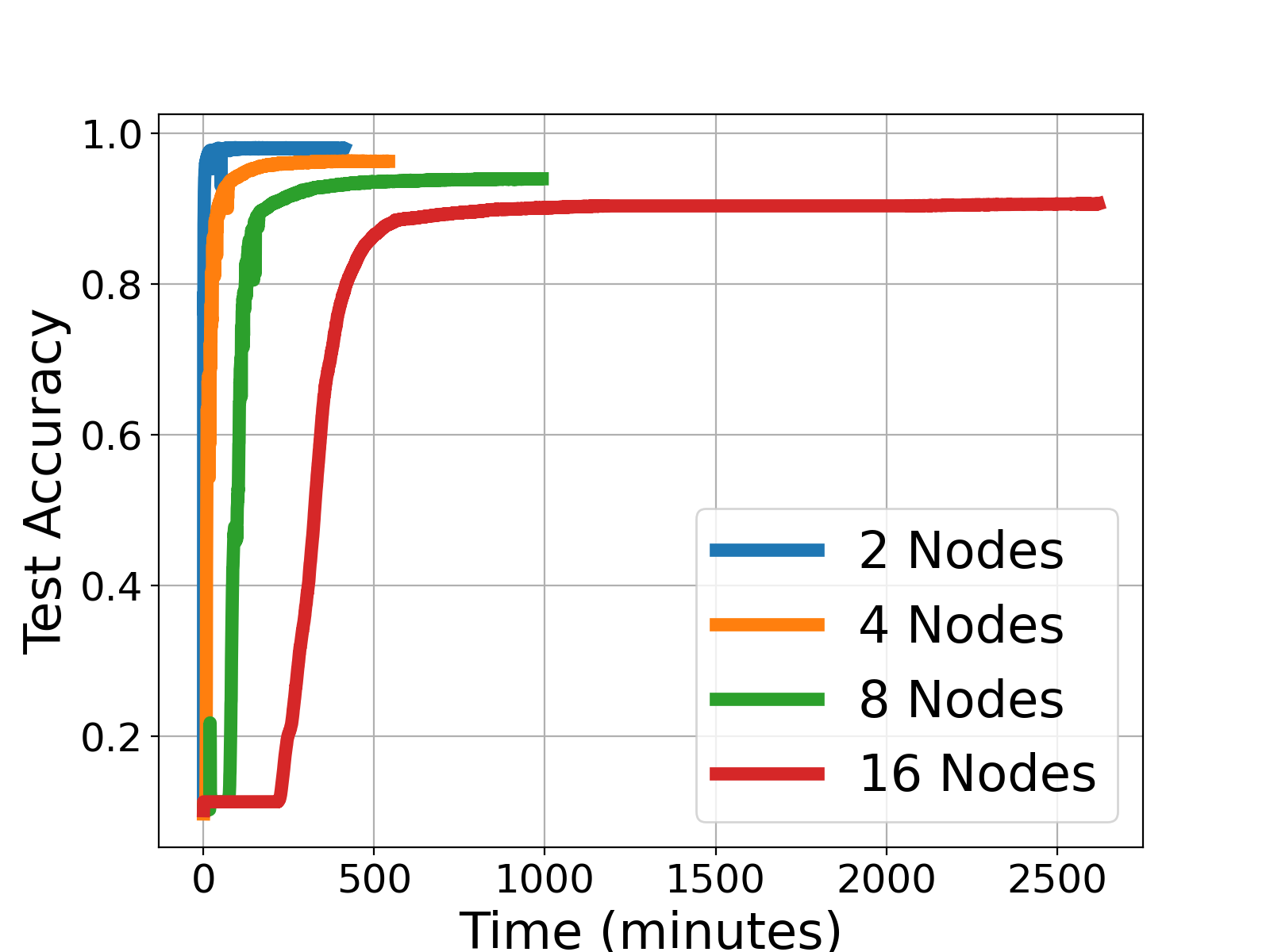}
        \label{fig1:d}
    \end{subfigure}
    \caption{Convergence results for different number of nodes}
    \label{fig1}
\end{figure}

\textbf{Graph Connectivity} This part experimentally studies the behaviour of our algorithm for different percentages of graph connectivity. We fixed the number of nodes to $16$ and did experiments with $\{0.5, 0.7, 0.9\}$ graph connectivity percentages. Figure \ref{fig2} shows the time-wise convergence results for different graph connectivity percentages. Our observation is that as the graph topology gets more connected, the convergence results get improve and our algorithm finds a better optimal parameters. 

\begin{figure}[H]
    \centering
    \begin{subfigure}[b]{0.40\textwidth}
        \includegraphics[width=1\linewidth]{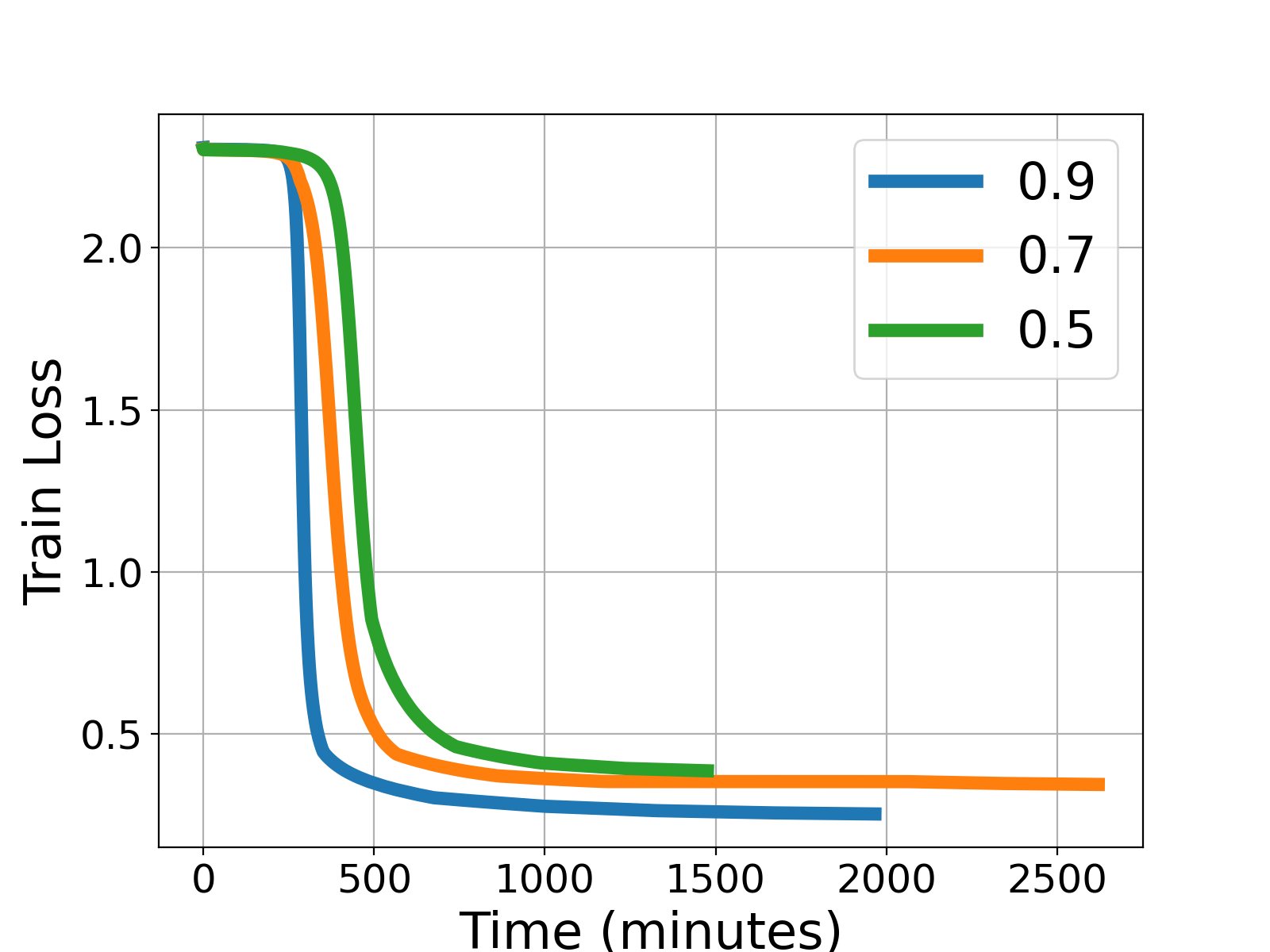}
        \label{fig2:a}
    \end{subfigure}
    \begin{subfigure}[b]{0.40\textwidth}
        \includegraphics[width=1\linewidth]{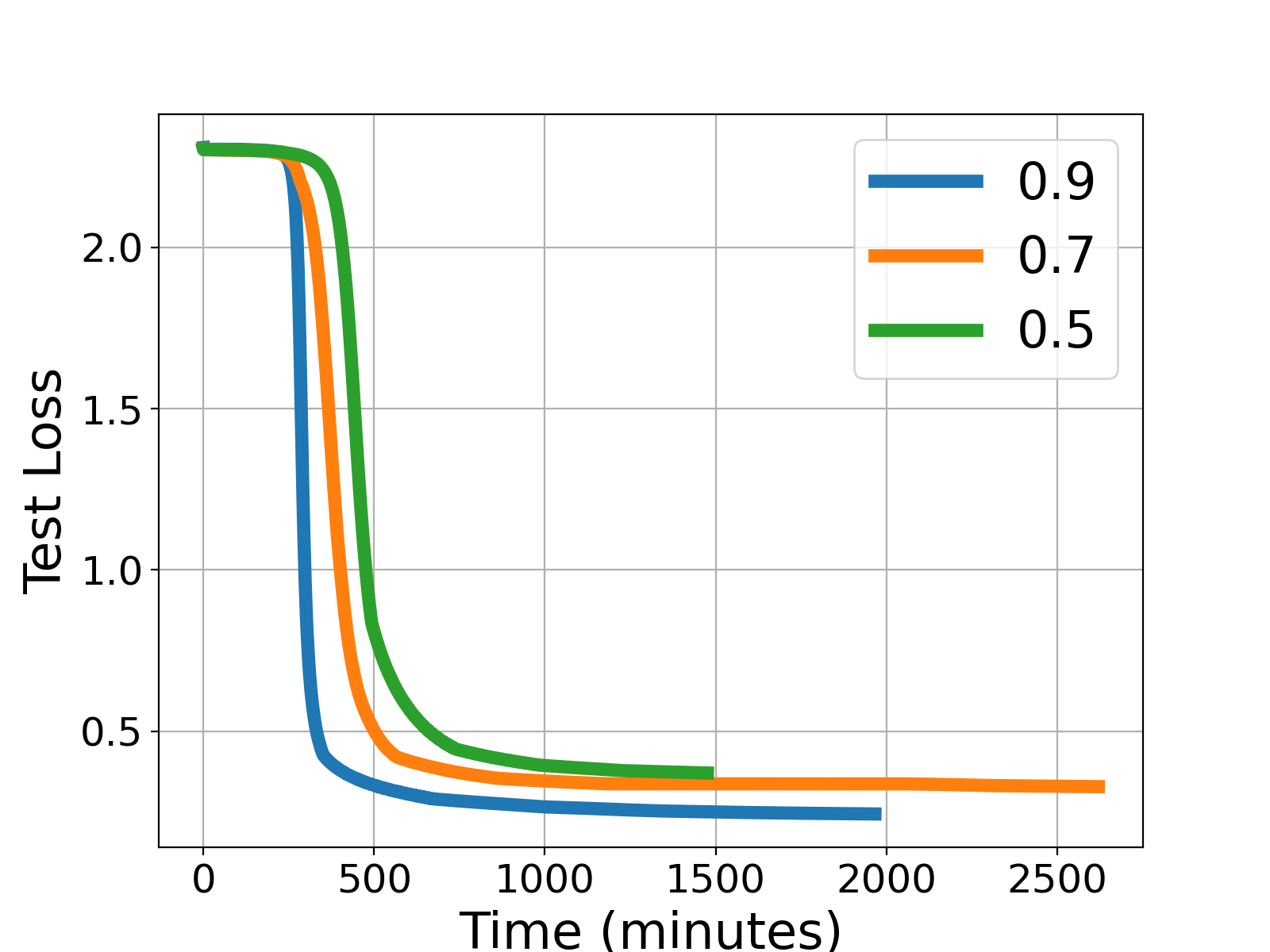}
        \label{fig2:b}
    \end{subfigure}
    \\ \vspace{0.1in}
    \begin{subfigure}[b]{0.40\textwidth}
        \includegraphics[width=1\linewidth]{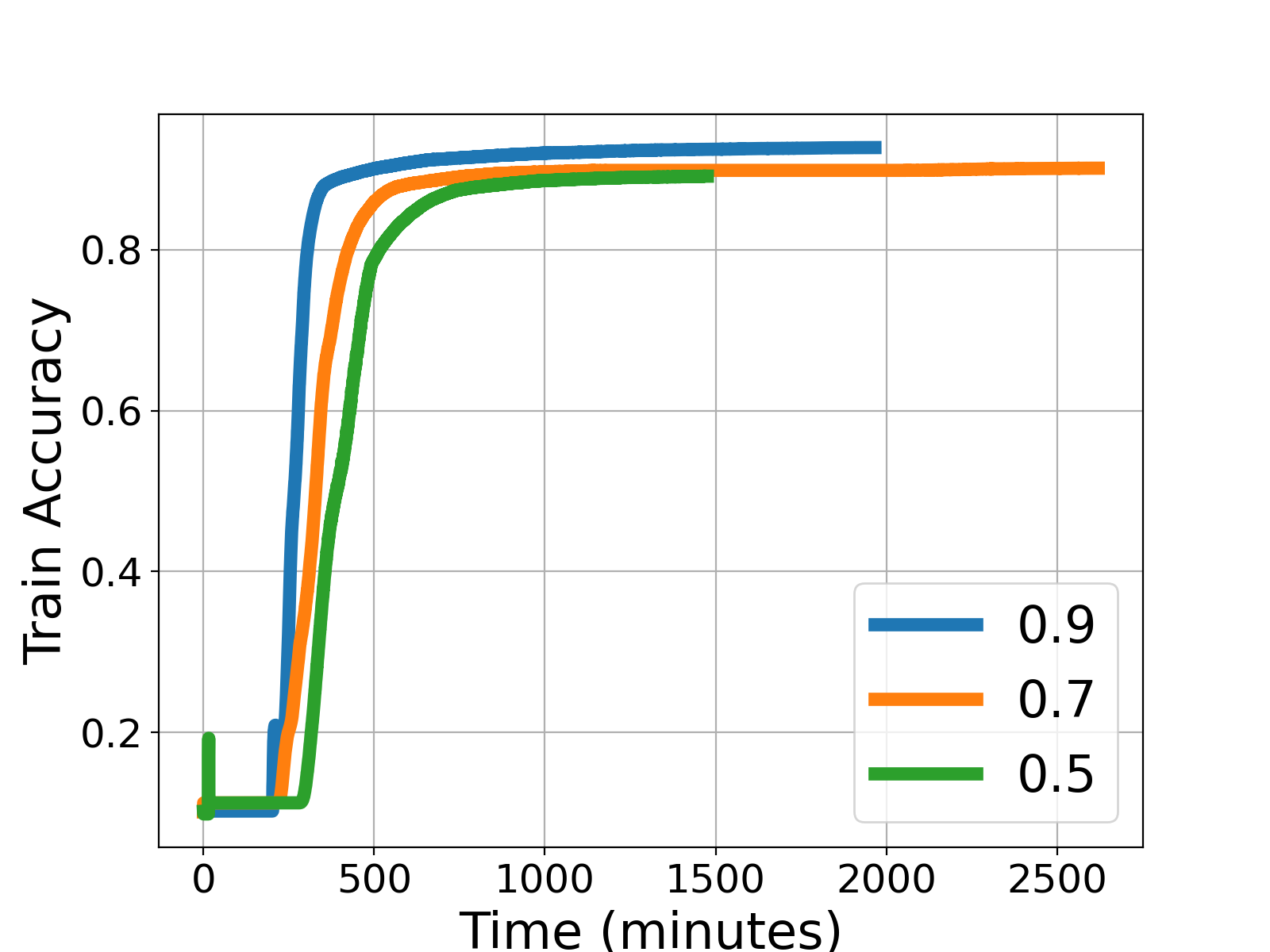}
        \label{fig2:c}
    \end{subfigure}
    \begin{subfigure}[b]{0.40\textwidth}
        \includegraphics[width=1\linewidth]{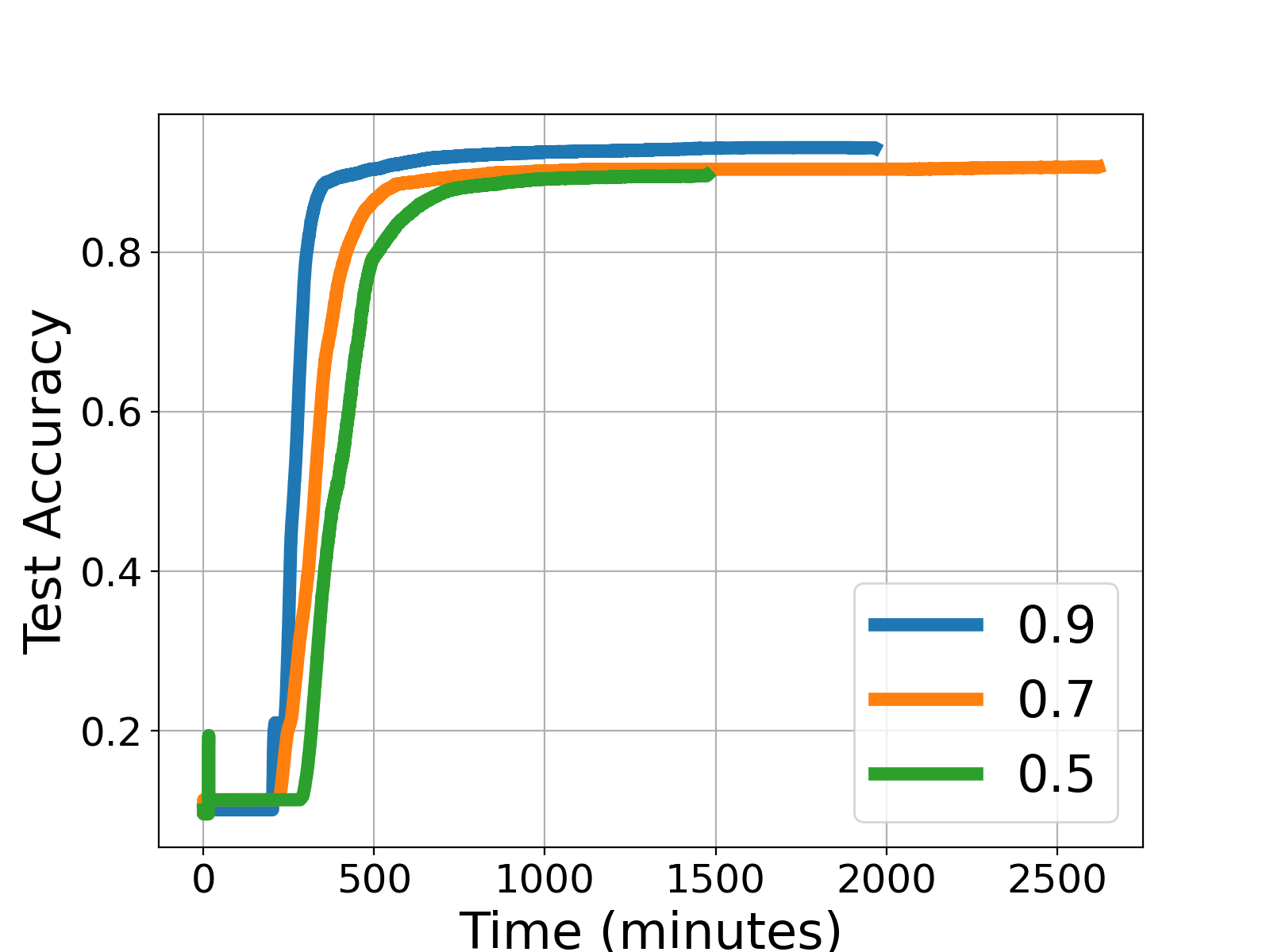}
        \label{fig2:d}
    \end{subfigure}
    \caption{Time-wise convergence results for different connectivity}
    \label{fig2}
\end{figure}

\textbf{Parameter deviations and Norm of gradients}

Figure \ref{fig3} shows the average $L_{\infty}$ distance of each node's parameters and the node-wise average, i.e. $\frac{1}{K} \sum_{i=1}^K ||x_i - x^{avg}||_{\infty}$ at the end of each snapshot, for two different graph connectivity percentages, i.e. $(0.5, 0.9)$. We observe that each node's parameters are approximately equidistant from the average. Moreover, there is a gradual increase at the beginning time around $500, 300$ minutes for $0.5, 0.9$ graph connectivity percentages respectively which is due to the initial learning-rate warm-up. After that point, we can see multiple reductions at the learning-rate reduction intervals. In fact, it is noticeable that the parameter deviations are smaller for higher connectivity percentages. 

Figure \ref{fig4} presents the norm infinity of gradients on the whole dataset using the node-wise averaged parameters. As the graph connectivity percentage increases, norm infinity of the gradients gets smaller and reduces faster. We can see a gradual increase at the beginning time around $500, 300$ minutes for $0.5, 0.9$ graph connectivity percentages respectively which is consistent with our observation in parameter deviation plots in Figure \ref{fig3}. 

\begin{figure}[t] 
    \centering
    \begin{subfigure}[b]{0.40\textwidth}
        \includegraphics[width=1\linewidth]{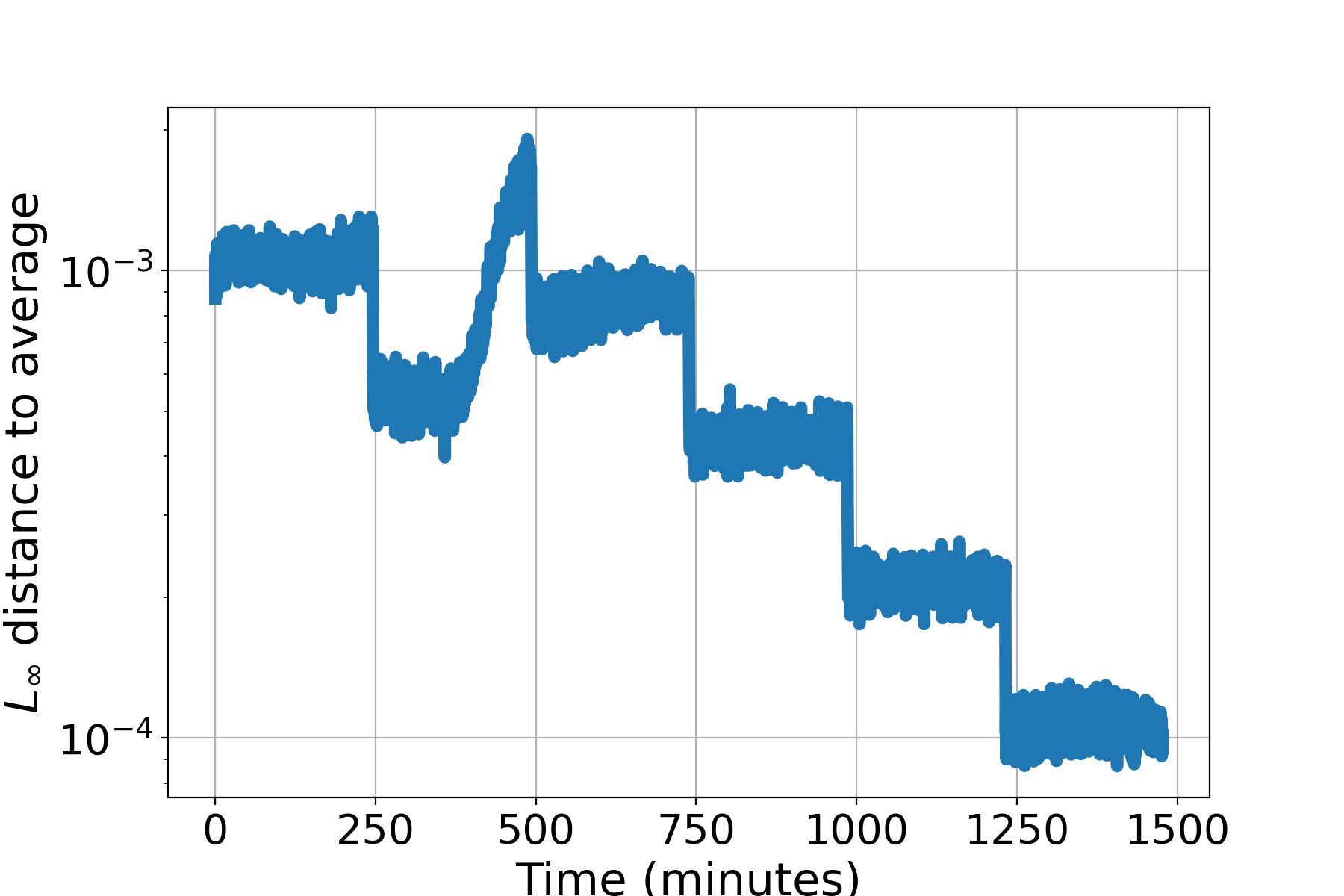}
        \label{fig3:a}
        \caption{0.5 graph connectivity}
    \end{subfigure}
    \begin{subfigure}[b]{0.40\textwidth}
        \includegraphics[width=1\linewidth]{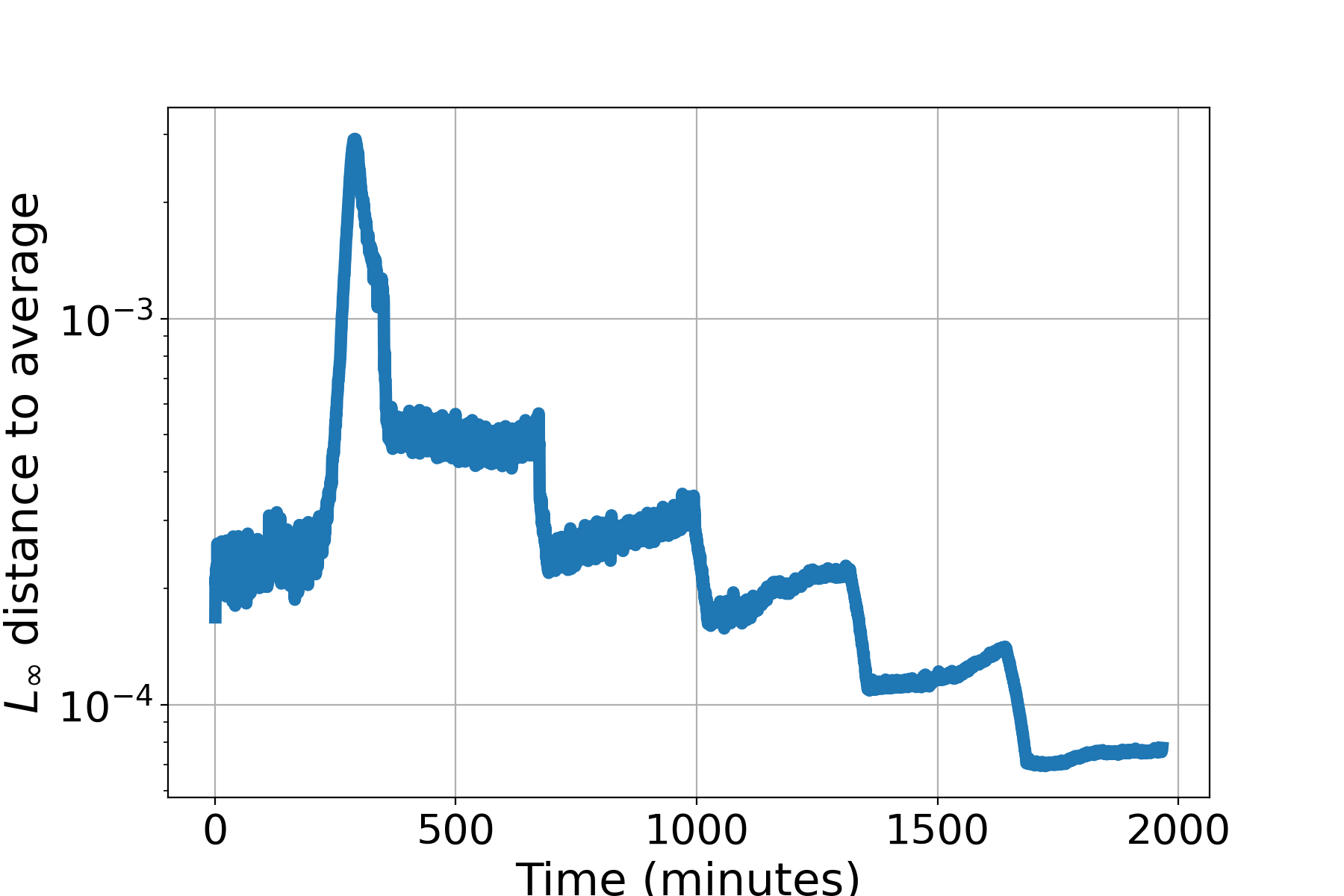}
        \label{fig3:b}
        \caption{0.9 graph connectivity}
    \end{subfigure}
    \caption{$L_{\infty}$ distance to average for different graph connectivity percentages}
    \label{fig3}
\end{figure}

\begin{figure}[H]
    \centering
    \begin{subfigure}[b]{0.40\textwidth}
        \includegraphics[width=1\linewidth]{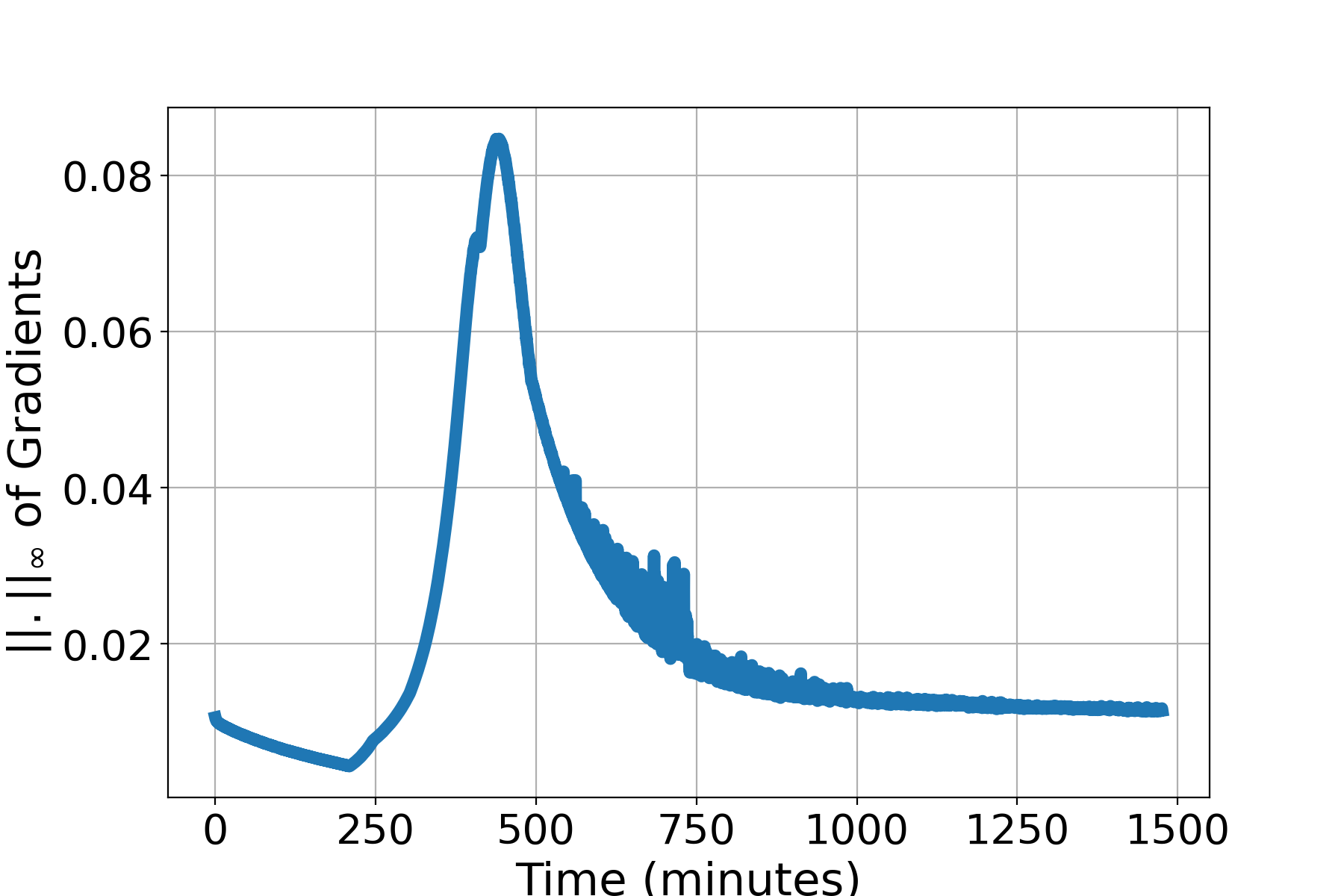}
        \caption{0.5 graph connectivity}
        \label{fig4:a}
    \end{subfigure}
    \begin{subfigure}[b]{0.40\textwidth}
        \includegraphics[width=1\linewidth]{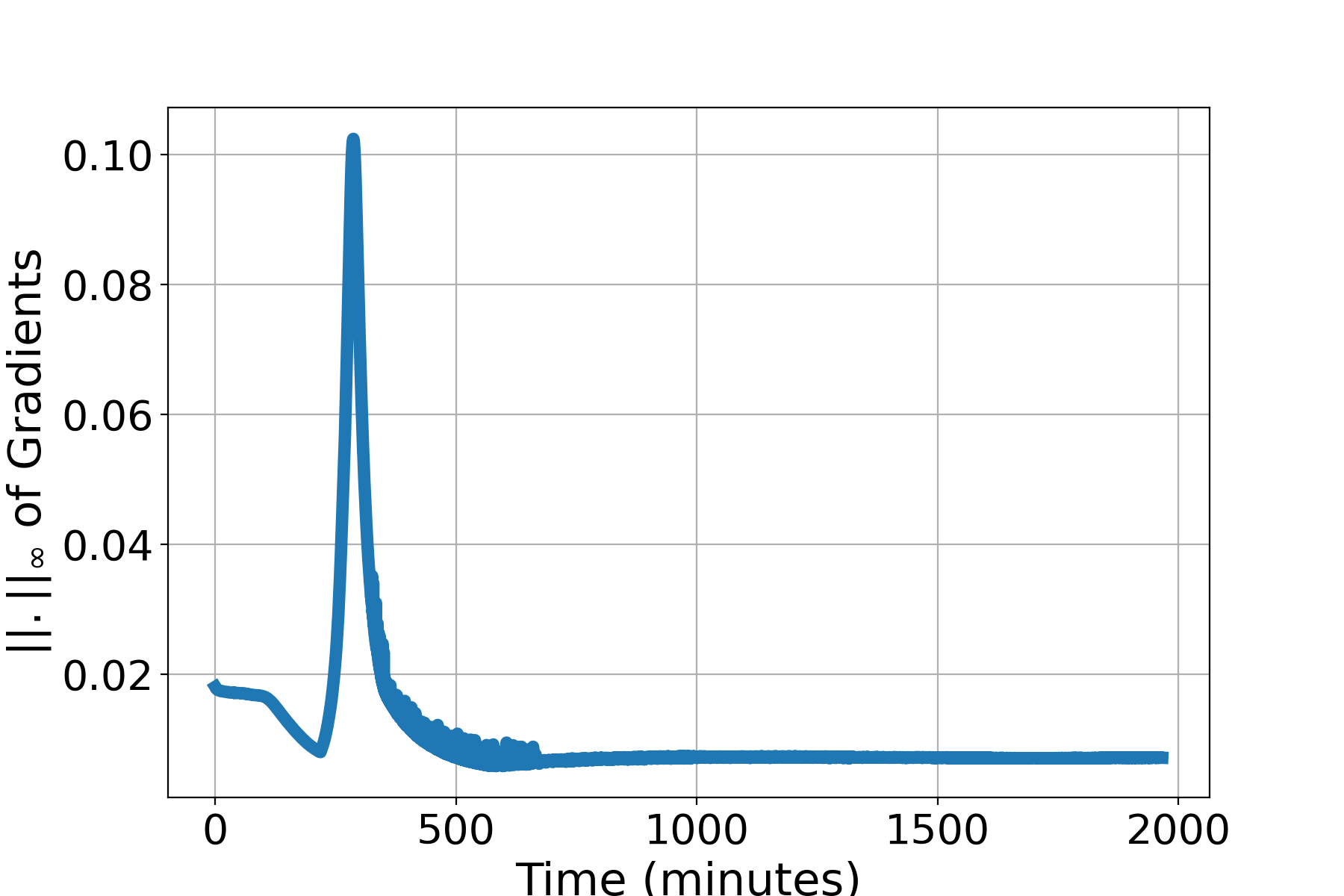}
        \caption{0.9 graph connectivity}
        \label{fig4:b}
    \end{subfigure}
    \caption{Norm of gradients for different graph connectivity percentages}
    \label{fig4}
\end{figure} 

\textbf{Maximum delay and Time per iteration}

Figures \ref{fig5:a},\ref{fig5:b} show maximum delay, and average time per iteration as a function of number of nodes for a fixed $0.7$ graph connectivity percentage respectively. Maximum delay is the delay between the fastest and slowest nodes. We observe that maximum delay increased as the number of nodes increased and it was always bounded. Indeed, from figure \ref{fig5:b} we can see that average time per iteration also increased as the number of nodes increased. Figure \ref{fig6} represents the behaviour of maximum delay w.r.t graph connectivity percentage. We can observe that maximum delay has been increased by increasing connectivity percentage and it was always bounded. 

\begin{figure}[H]
    \centering
    \begin{subfigure}[b]{0.40\textwidth}
        \includegraphics[width=1\linewidth]{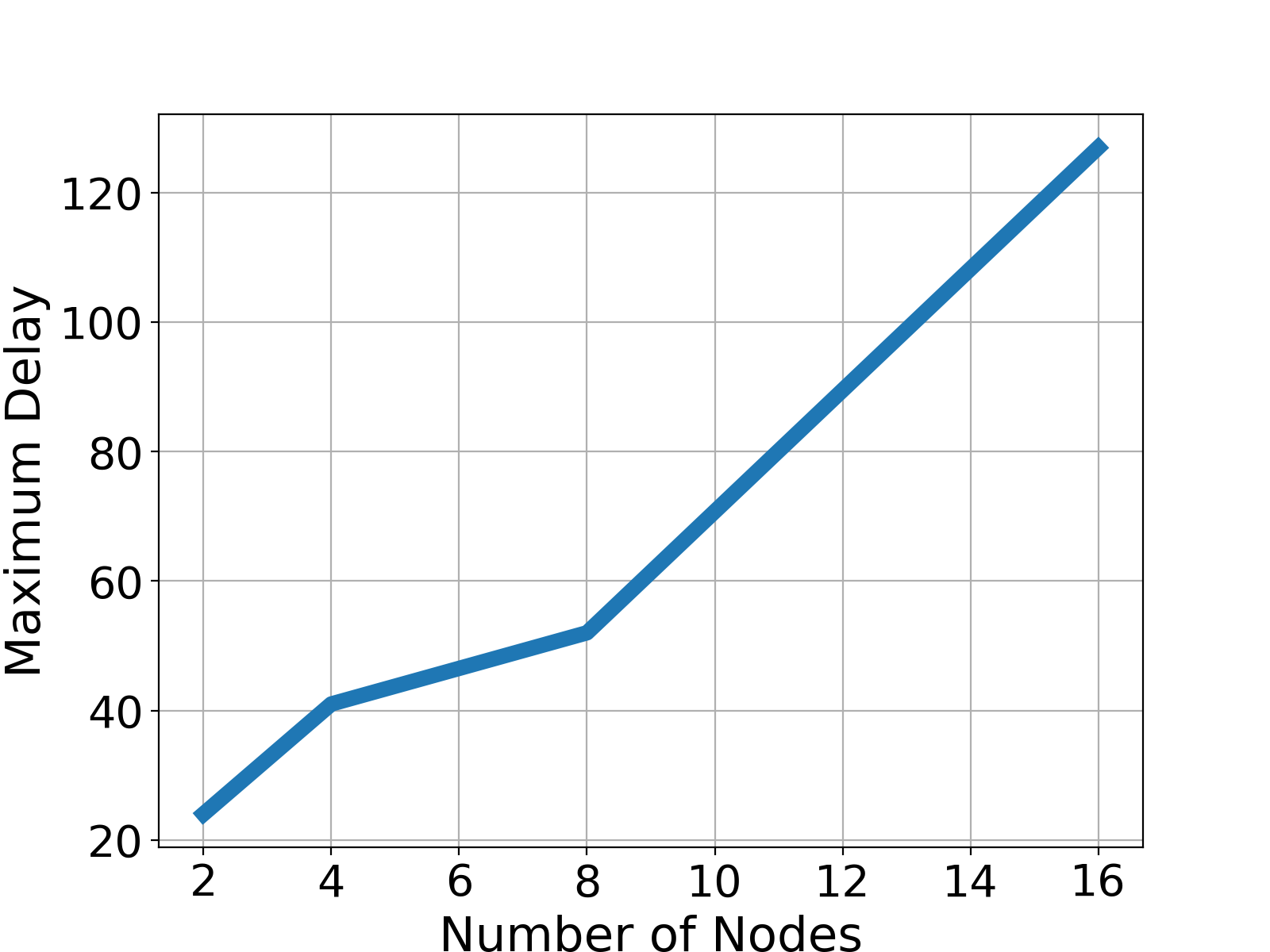}
        \caption{Maximum Delay}\label{fig5:a}
    \end{subfigure}
    \begin{subfigure}[b]{0.40\textwidth}
        \includegraphics[width=1\linewidth]{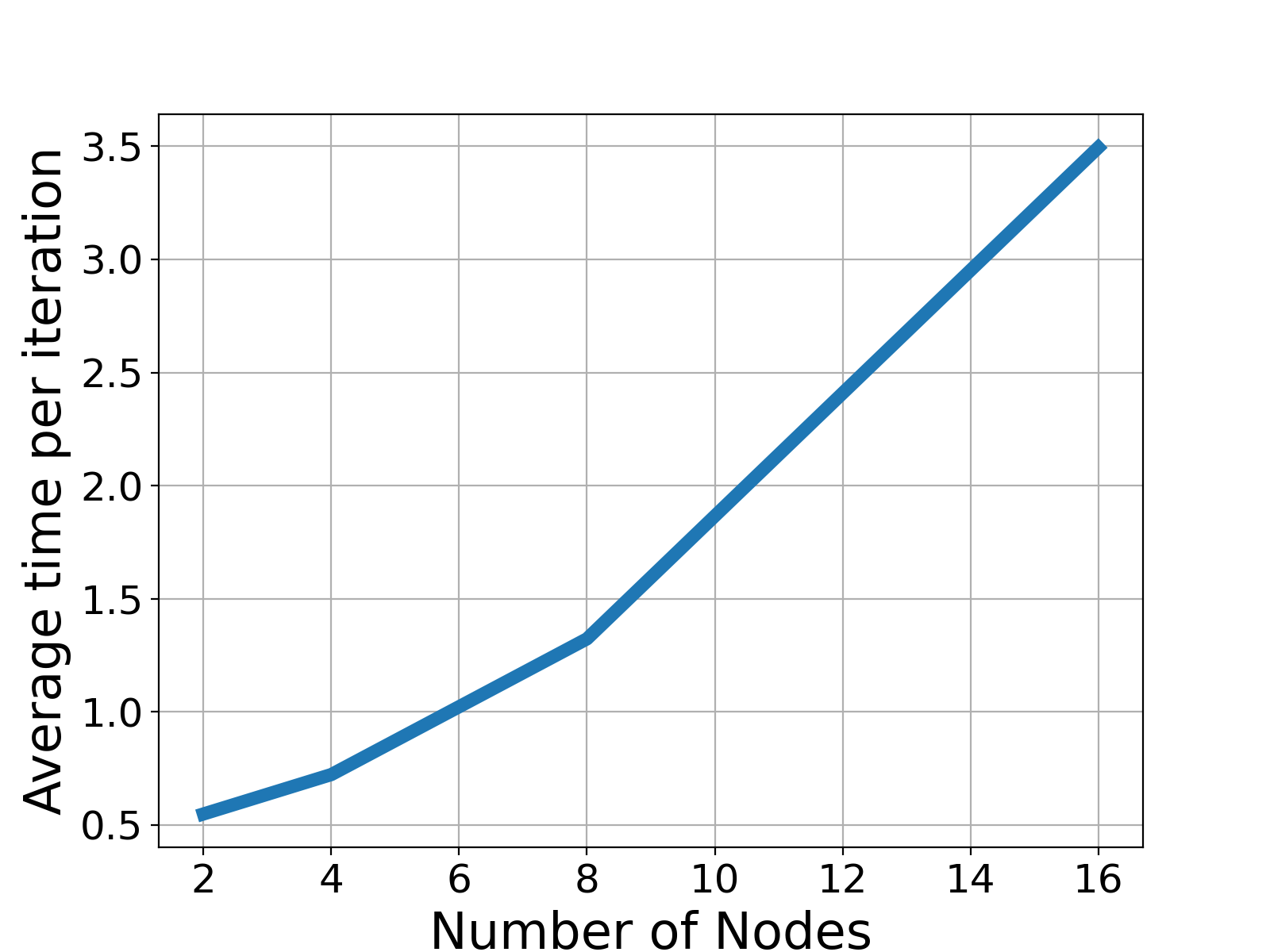}\caption{Average Time Per Iteration}
        \label{fig5:b}
    \end{subfigure}
    \caption{Maximum delay and average time per iteration across different number of nodes for $0.7$ graph connectivity percentage}
    \label{fig5}
\end{figure}

\begin{figure}[H]
    \centering
    \begin{subfigure}[b]{0.40\textwidth}
        \includegraphics[width=1\linewidth]{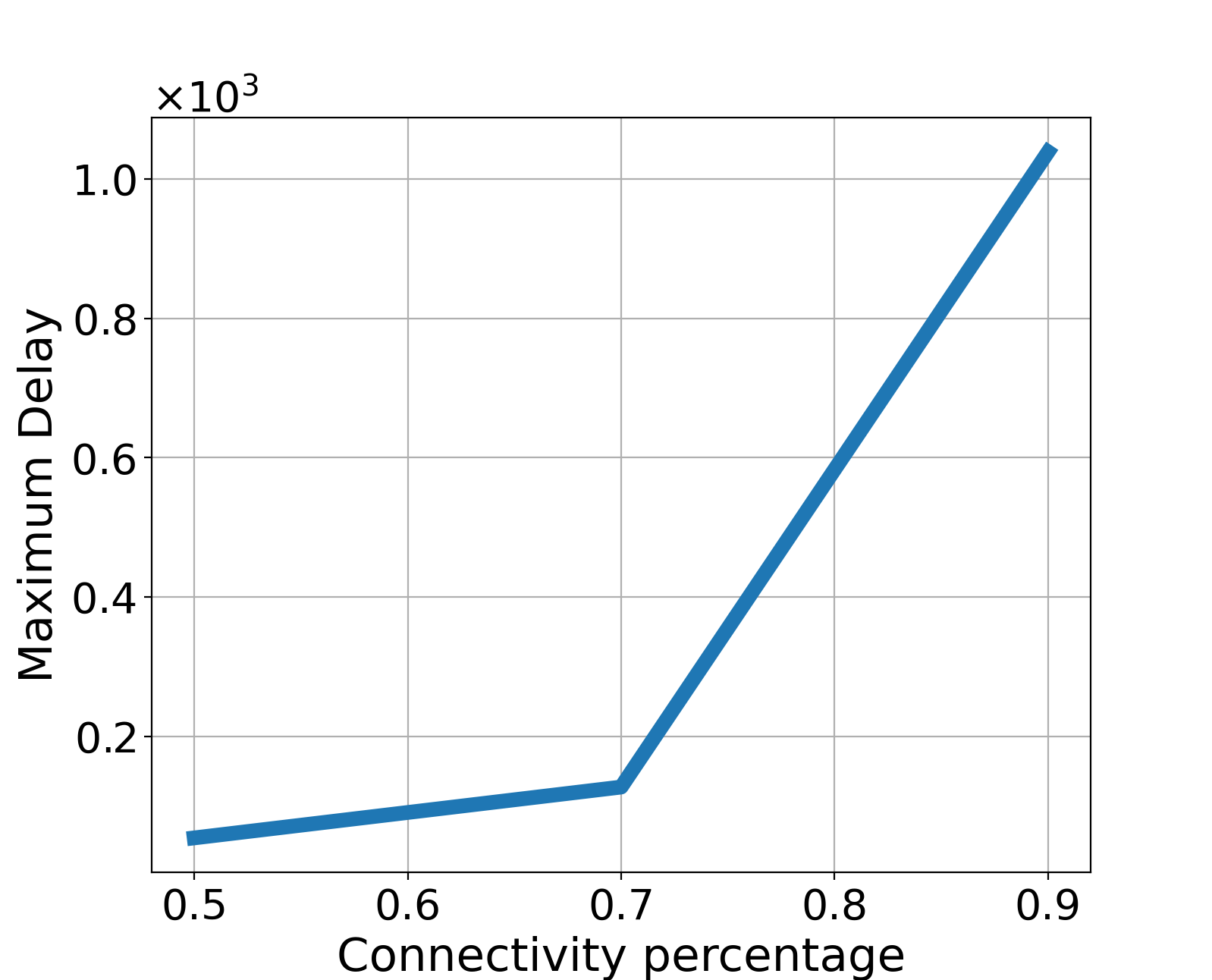}
        
    \end{subfigure}
    \caption{Maximum delay w.r.t graph connectivity percentage}
    \label{fig6}
\end{figure}

\section{Conclusion}
In this paper we have studied stochastic nonconvex decentralized optimization on directed graphs with asynchronous communication, closing an important gap in the literature on distributed optimization. The theoretical results confirm the expected sublinear convergence rate, and corroborate a similar pattern of faster consensus and tracking convergence, leaving the optimization to dominate the error asymptotically. Our numerical results confirm the convergence of the algorithm and show the scalability with the number of nodes and different graph connectivity percentages on a non-convex image classification task. 

\section*{Acknowledgment}
VK acknowledges support to the OP VVV funded project CZ.02.1.01/0.0/0.0/16\_019/0000765 ``Research Center for Informatics''
\bibliographystyle{IEEEtran}
\bibliography{refs}

\begin{thebibliography}{10}
\providecommand{\url}[1]{#1}
\csname url@samestyle\endcsname
\providecommand{\newblock}{\relax}
\providecommand{\bibinfo}[2]{#2}
\providecommand{\BIBentrySTDinterwordspacing}{\spaceskip=0pt\relax}
\providecommand{\BIBentryALTinterwordstretchfactor}{4}
\providecommand{\BIBentryALTinterwordspacing}{\spaceskip=\fontdimen2\font plus
\BIBentryALTinterwordstretchfactor\fontdimen3\font minus
  \fontdimen4\font\relax}
\providecommand{\BIBforeignlanguage}[2]{{%
\expandafter\ifx\csname l@#1\endcsname\relax
\typeout{** WARNING: IEEEtran.bst: No hyphenation pattern has been}%
\typeout{** loaded for the language `#1'. Using the pattern for}%
\typeout{** the default language instead.}%
\else
\language=\csname l@#1\endcsname
\fi
#2}}
\providecommand{\BIBdecl}{\relax}
\BIBdecl

\bibitem{tsitsiklis1986distributed}
J.~Tsitsiklis, D.~Bertsekas, and M.~Athans, ``Distributed asynchronous
  deterministic and stochastic gradient optimization algorithms,'' \emph{IEEE
  transactions on automatic control}, vol.~31, no.~9, pp. 803--812, 1986.

\bibitem{lian2015asynchronous}
X.~Lian, Y.~Huang, Y.~Li, and J.~Liu, ``Asynchronous parallel stochastic
  gradient for nonconvex optimization,'' in \emph{Advances in Neural
  Information Processing Systems}, 2015, pp. 2737--2745.

\bibitem{peng2016arock}
Z.~Peng, Y.~Xu, M.~Yan, and W.~Yin, ``Arock: an algorithmic framework for
  asynchronous parallel coordinate updates,'' \emph{SIAM Journal on Scientific
  Computing}, vol.~38, no.~5, pp. A2851--A2879, 2016.

\bibitem{agarwal2011distributed}
A.~Agarwal and J.~C. Duchi, ``Distributed delayed stochastic optimization,'' in
  \emph{Advances in Neural Information Processing Systems}, 2011, pp. 873--881.

\bibitem{mcmahan2017communication}
B.~McMahan, E.~Moore, D.~Ramage, S.~Hampson, and B.~A. y~Arcas,
  ``Communication-efficient learning of deep networks from decentralized
  data,'' in \emph{Artificial Intelligence and Statistics}, 2017, pp.
  1273--1282.

\bibitem{xin2019variance}
R.~Xin, U.~A. Khan, and S.~Kar, ``Variance-reduced decentralized stochastic
  optimization with gradient tracking,'' \emph{arXiv preprint
  arXiv:1909.11774}, 2019.

\bibitem{tang2018d}
H.~Tang, X.~Lian, M.~Yan, C.~Zhang, and J.~Liu, ``D2: Decentralized training
  over decentralized data,'' in \emph{International Conference on Machine
  Learning}, 2018, pp. 4848--4856.

\bibitem{lian2017asynchronous}
X.~Lian, W.~Zhang, C.~Zhang, and J.~Liu, ``Asynchronous decentralized parallel
  stochastic gradient descent,'' \emph{arXiv preprint arXiv:1710.06952}, 2017.

\bibitem{ConvexAsynchDecSGD}
T.~Wu, K.~Yuan, Q.~Ling, W.~Yin, and A.~H. Sayed, ``Decentralized consensus
  optimization with asynchrony and delays,'' \emph{IEEE Transactions on Signal
  and Information Processing over Networks}, vol.~4, no.~2, pp. 293--307, 2018.

\bibitem{tsianos2012push}
K.~I. Tsianos, S.~Lawlor, and M.~G. Rabbat, ``Push-sum distributed dual
  averaging for convex optimization,'' in \emph{2012 ieee 51st ieee conference
  on decision and control (cdc)}.\hskip 1em plus 0.5em minus 0.4em\relax IEEE,
  2012, pp. 5453--5458.

\bibitem{nedic2010convergence}
A.~Nedi{\'c} and A.~Ozdaglar, ``Convergence rate for consensus with delays,''
  \emph{Journal of Global Optimization}, vol.~47, no.~3, pp. 437--456, 2010.

\bibitem{olshevsky2018robust}
A.~Olshevsky, I.~C. Paschalidis, and A.~Spiridonoff, ``Robust asynchronous
  stochastic gradient-push: Asymptotically optimal and network-independent
  performance for strongly convex functions,'' \emph{arXiv preprint
  arXiv:1811.03982}, 2018.

\bibitem{tian2018achieving}
Y.~Tian, Y.~Sun, and G.~Scutari, ``Achieving linear convergence in distributed
  asynchronous multi-agent optimization,'' \emph{IEEE Trans. on Automatic
  Control}, 2020.

\bibitem{zhang2019fully}
J.~Zhang and K.~You, ``Fully asynchronous distributed optimization with linear
  convergence in directed networks,'' \emph{arXiv preprint arXiv:1901.08215},
  2019.

\bibitem{zhang2019decentralized}
------, ``Decentralized stochastic gradient tracking for empirical risk
  minimization,'' \emph{arXiv preprint arXiv:1909.02712}, 2019.

\bibitem{pu2019sharp}
S.~Pu, A.~Olshevsky, and I.~Paschalidis, ``A sharp estimate on the transient
  time of distributed stochastic gradient descent,'' \emph{arXiv preprint
  arXiv:1906.02702}, 2019.

\bibitem{ghadimi2013stochastic}
S.~Ghadimi and G.~Lan, ``Stochastic first-and zeroth-order methods for
  nonconvex stochastic programming,'' \emph{SIAM Journal on Optimization},
  vol.~23, no.~4, pp. 2341--2368, 2013.

\bibitem{lecun2010mnist}
Y.~LeCun, C.~Cortes, and C.~Burges, ``Mnist handwritten digit database,''
  \emph{ATT Labs [Online]. Available: http://yann.lecun.com/exdb/mnist},
  vol.~2, 2010.

\bibitem{Tensorflow-MNIST}
\BIBentryALTinterwordspacing
T.~Tutorials. https://www.tensorflow.org/tutorials/quickstart/advanced.
  [Online]. Available:
  \url{"https://www.tensorflow.org/tutorials/quickstart/advanced"}
\BIBentrySTDinterwordspacing

\end{thebibliography}
\end{document}